\documentclass[12pt]{amsart}
\usepackage[margin=1in]{geometry}
\usepackage[utf8]{inputenc}
\usepackage{amsmath}
\usepackage{amssymb}
\usepackage{amscd}
\usepackage{color}
\usepackage{amsthm, graphicx, subfig, float, tikz}
\usepackage{ulem}
\usepackage{hyperref}

\newtheorem{theorem}{Theorem}[section]
\newtheorem{lemma}[theorem]{Lemma}

\newtheorem{definition}[theorem]{Definition}

\theoremstyle{remark}

\newtheorem*{problem}{Problem}

\newcommand{\R}{\mathbb{R}}

\newcommand{\pdev}[2]{\frac{\partial}{\partial #1} (#2)}
\newcommand{\spdev}[2]{\frac{\partial #2}{\partial #1}}

\DeclareMathOperator{\disc}{disc}
\DeclareMathOperator{\Hess}{Hess}
\DeclareMathOperator{\Real}{Re}
\DeclareMathOperator{\Imag}{Im}
\DeclareMathOperator{\Ker}{Ker}
\newcommand{\rvec}{{\mathbf{r}}}
\newcommand{\pvec}{{\mathbf{p}}}
\newcommand{\vr}{v}
\newcommand{\Diff}{\mathrm{D}}

\title[Locating conical degeneracies in the spectrum]{Locating conical
  degeneracies in the
  spectra of parametric self-adjoint matrices}

\author{G.~Berkolaiko} 
\address{Department of Mathematics,
  Texas A\&M University, College Station, TX 77843-3368, USA}
\author{A.~Parulekar}
\address{Department of Electrical and Computer Engineering,
  University of Texas at Austin
  Austin, TX 78712, USA}

\begin{document}

\begin{abstract}
  A simple iterative scheme is proposed for locating the parameter
  values for which a 2-parameter family of real symmetric matrices has
  a double eigenvalue.  The convergence is proved to be quadratic.  An
  extension of the scheme to complex Hermitian matrices (with 3
  parameters) and to location of triple eigenvalues (5 parameters for
  real symmetric matrices) is also described.  Algorithm convergence
  is illustrated in several examples: a real symmetric family, a
  complex Hermitian family, a family of matrices with an ``avoided
  crossing'' (no convergence) and a 5-parameter family of real
  symmetric matrices with a triple eigenvalue.
\end{abstract}

\maketitle

\section{Introduction}

A theorem of von Neumann and Wigner states that, generically, a
two-parameter family of real symmetric matrices has multiple
eigenvalues at isolated points \cite{WigVon}. In other words, the
matrices with multiple eigenvalues have co-dimension 2 in the manifold
of real symmetric matrices \cite[Appendix 10]{Arnold_mechanics}.  In
this paper, we would like to address the problem of locating these
isolated points of eigenvalue multiplicity in the 2-dimensional
parameter space.  To be more precise, we consider the following
problem.
\begin{problem}
  Given a smooth real symmetric matrix valued function
  $A: \mathbb{R}^2 \mapsto \mathbb{R}^{n \times n}$, locate the
  values of the parameters $(x,y)$ which yield a matrix $A(x,y)$ with
  degenerate eigenvalues.
\end{problem}

To give a simple example, the function
\begin{equation*}
  A(x,y) = 
  \begin{pmatrix}
    x & y \\
    y & -x
  \end{pmatrix}
\end{equation*}
has a double eigenvalue at the unique point $(x,y)=(0,0)$.  Its
eigenvalues $\lambda$ satisfy the equation $\lambda^2 = x^2+y^2$ and
the eigenvalue surface is a circular double cone in the space
$(x,y,\lambda)$.  In contrast, the nonlinear function
\begin{equation}
  \label{eq:example_func}
  A(x,y) = 
  \begin{pmatrix} 
    \cos(y)\sin(x) & 2-3\sin(y-x) \\
    2-3\sin(y-x) & 2\cos(y)-\sin(x) 
  \end{pmatrix}  
\end{equation}
has multiple points of eigenvalue multiplicity, see
Figure~\ref{fig:simple_cone}.  Each point is isolated and locally
around each point the eigenvalue surface also looks like a cone.

For a family of complex Hermitian matrices, the co-dimension of the
matrices with multiple eigenvalues is 3.  Therefore, the analogous
question can be posed about locating multiple eigenvalues of a
Hermitian $A(x,y,z)$.  We will formulate an extension of our results
to complex Hermitian matrices but will concentrate on the real
symmetric case in our proofs.

\begin{figure}
  \centering
  \includegraphics[scale=0.5]{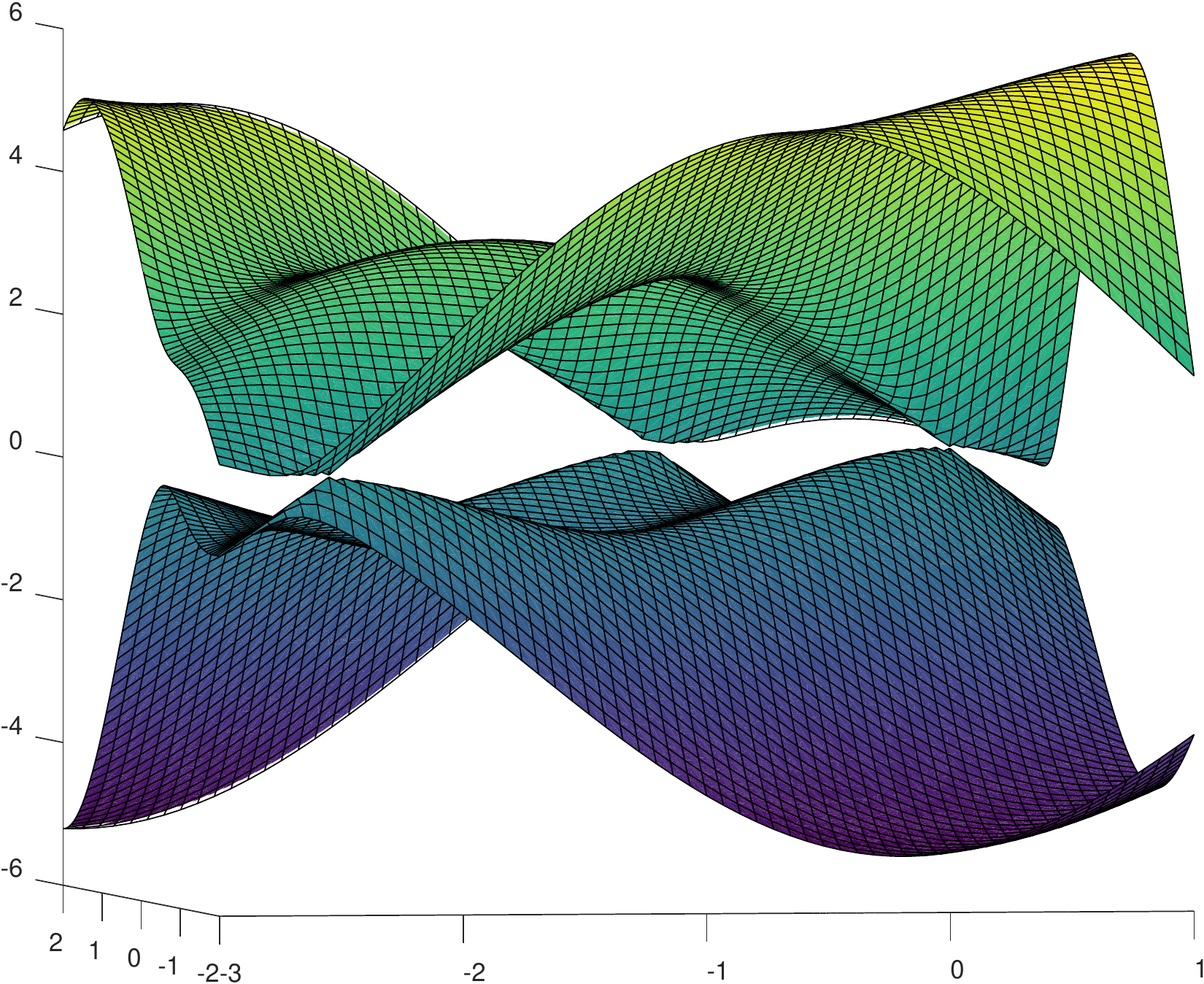}
  \caption{Eigenvalue surfaces corresponding to $A(x,y)$ from
    \eqref{eq:example_func}.  There are three conical point; the
    surfaces appear not to touch at the middle point due to
    insufficient grid precision.}
  \label{fig:simple_cone}
\end{figure}

The problem of locating the points of eigenvalue multiplicity is of
practical importance.  In condensed matter physics
\cite{AshcroftMermin_solid} the wave propagation through periodic
medium is studied via Floquet--Bloch transform
\cite{Kuc_floquet,Kuc_bams16} which results in a parametric family of
self-adjoint operators (or matrices) with discrete spectrum.  The
eigenvalue surfaces (sheets of the ``dispersion relation'') may touch,
see Fig.~\ref{fig:simple_cone}, which has profound effect on wave
propagation and its sensitivity to a small perturbation of the medium.
This touching corresponds precisely to a multiplicity in the
eigenvalue spectrum.  To give a well-studied example, the unusual
electron properties of graphene occur due to the presence of
eigenvalue multiplicity \cite{CasGraphen_rmp09,Nov_nob10}.  It is also
of practical relevance to be able to distinguish touching from
``almost touching'' (also known as ``avoided crossing'' in
one-parameter problems).

The question of locating eigenvalue multiplicity in a family of
$2\times 2$ real symmetric matrices $A$ has a straightforward solution
(which also illustrates why the co-dimension is 2).  The discriminant
of $A \in \mathbb{R}^{2\times 2}$ can be written as a sum of two
squares,
\begin{equation}
  \label{eq:discriminant_def}
  \disc(A):=(\lambda_1-\lambda_2)^2=(A_{11}-A_{22})^2+4A_{12}^2.  
\end{equation}
By definition, the discriminant is $0$ if and only if two eigenvalues
coincide, therefore we have two conditions that must simultaneously be
met for the multiplicity to occur:
\begin{equation}
  \label{eq:F_def}
  F(x,y) = \textbf{0}, 
  \qquad
  \mbox{where}\quad
  F : \R^2 \to \R^2, \quad
  \ F(x,y) :=
  \begin{pmatrix}
    A_{11}(x,y)-A_{22}(x,y)\\
    A_{12}(x,y)
  \end{pmatrix}.
\end{equation}
Unfortunately, for larger matrices the discriminant quickly becomes
unwieldy and cannot be used in practical computations.  The
discriminant can still be written as a sum of squares
\cite{Ily,Lax,Par_laa02,DanIkr}, but the number of terms grows fast
with the size of the matrix.

Thus, for an $n\times n$ real symmetric matrix $A(x,y)$ depending on
two parameters $x$ and $y$ there is only one easily computable
function $\lambda_2(x,y) - \lambda_1(x,y)$ whose root, in variables
$x$ and $y$, we are seeking.\footnote{Here, without loss of
  generality, we have assumed that one is interested in the degeneracy
  $\lambda_1 = \lambda_2 < \lambda_3 < \ldots$} However, to apply a
standard method with quadratic convergence, such as the
Newton--Raphson algorithm, one needs 2 functions for 2 variables.  One
can search for the minimum of the square eigenvalue difference,
$\big(\lambda_2(x,y) - \lambda_1(x,y)\big)^2$, which is smooth.  But
such a search would converge equally well to a point of ``avoided
crossing'', a pitfall our proposed method manages to avoid, see
Sections~\ref{sec:avoided} and \ref{sec:mergingDirac}.

One can change the basis to make $A(x,y)$ block-diagonal, with a
$2\times2$ block corresponding to eigenvalues $\lambda_1$ and
$\lambda_2$.  The existence of this change in a neighborhood of the
multiplicity point is assured (using Riesz projector) if
$\lambda_{1,2}$ remain bounded away from the rest of the spectrum.
However the new basis will depend on the parameters $(x,y)$ and is not
directly accessible for numerical computations.  Despite this
obstacle, we will show that a ``naive'' approach produces equivalently
good convergence: one can use a \emph{constant} eigenvector basis
which is recomputed\footnote{We are motivated mostly by the
  applications to tight-binding models of condensed matter physics
  \cite{AshcroftMermin_solid} where the matrix dimesion $n$ is often
  of order $10$ and computation of eigenvectors is relatively fast and
  precise.  Another area of application is pointed out at the end of
  Section~\ref{sec:mergingDirac}.} at each point of the
Newton--Raphson iteration.  More precisely, we establish the following
theorem.

\begin{theorem}
  \label{thm:main}
  Let $A(\rvec):\mathbb{R}^2\mapsto \mathbb{R}^{n\times n}$
  be a real symmetric matrix valued function which is continuously
  twice differentiable in each entry, with a non-degenerate conical
  point (defined below) between $\lambda_1$ and $\lambda_2$ at
  parameter point $\alpha$. For any $\rvec_i$, define
  $\rvec_{i+1}$ by
  \begin{equation}
    \label{iter_equation_simple}
    \rvec_{i+1} = \rvec_i -
    \begin{pmatrix} 
      \big\langle \vr_1, \spdev{x}{A}\vr_1\big\rangle
      -\big\langle \vr_2, \spdev{x}{A}\vr_2\big\rangle 
      & \big\langle \vr_1, \spdev{y}{A}\vr_1\big\rangle
      -\big\langle \vr_2, \spdev{y}{A}\vr_2\big\rangle \\
      2\big\langle \vr_1, \spdev{x}{A}\vr_2\big\rangle 
      & 2\big\langle \vr_1, \spdev{y}{A}\vr_2\big\rangle
    \end{pmatrix}^{-1}
    \begin{pmatrix}
      \lambda_1 - \lambda_2 \\ 0
    \end{pmatrix}
  \end{equation}
  where 
  $\lambda_{1,2}=\lambda_{1,2}(\rvec_i)$ denote the eigenvalues of $A$
  at the point $\rvec_i$ and $\vr_{1,2}=\vr_{1,2}(\rvec_i)$ denote the
  corresponding eigenvectors.
  
  Then there exists an open neighborhood $\Omega\subset \mathbb{R}^2$ of
  $\alpha$ and a constant $C>0$ such that for all $\rvec_i\in \Omega$, the
  corresponding $\rvec_{i+1}$ satisfies the estimate
  \begin{equation}
    \label{eq:quadratic_conv}
    |\rvec_{i+1}-\alpha|<C|\rvec_i-\alpha|^2.
  \end{equation}
\end{theorem}


Before we prove this theorem in Section~\ref{sec:proofs}, we explain in
Section~\ref{sec:discussion} the geometrical picture behind the
iterative procedure~(\ref{iter_equation_simple}) and also point out
the main differences between~(\ref{iter_equation_simple}) and the
Newton--Raphson method in a conventional setting.  We also review
related literature in Section~\ref{sec:berry_phase} once we introduce
relevant notions.  The precise definition and properties of
``nondegenerate conical point'' is given in
Section~\ref{sec:conical_points}.  Section~\ref{sec:examples} contains
some computational examples.

\subsection{Notation}

We let $C^2(\mathbb{R}^2, \mathbb{R}^{n\times n})$ denote the set of
matrix valued functions mapping $\mathbb{R}^2$ to
$\mathbb{R}^{n\times n}$ with each element being continuously twice
differentiable. The eigenvalues of the matrix function $A\in
C^2(\mathbb{R}^2, \mathbb{R}^{n\times n})$ are numbered in the
increasing order 
$\lambda_1 \leq \lambda_2 \leq \lambda_3 \leq \cdots \leq \lambda_n$
and without loss of generality we will look for
$\rvec=(x,y)\in\mathbb{R}^2$ such that $\lambda_1(\rvec) =
\lambda_2(\rvec)$.  Naturally, all results apply equally well to any
pair of consecutive eigenvalues.  We remark that functions
$\lambda_k(\rvec)$ are continuous but not necessarily smooth: the
points of eigenvalue multiplicity are typically the points where the
eigenvalues involved are not differentiable, see
Fig.~\ref{fig:simple_cone}. 

For any real symmetric matrix
valued function $A$ and any point $\pvec\in \mathbb{R}^2$, we let
$A^\pvec = V^* A(\rvec) V$ denote
the representation of $A$ in the eigenvector basis computed at point
$\pvec$.  That is, $V$ is a fixed orthogonal matrix whose columns
are the eigenvectors of $A(\pvec)$.  The eigenvectors are assumed to
be numbered according to the eigenvalue ordering.  This means that 
$A^\pvec\in C^2(\mathbb{R}^2, \mathbb{R}^{n\times n})$ is a
diagonal matrix at the point $\pvec$ but not necessarily anywhere else.

We let
\begin{equation}
  \label{eq:tildeA_def}
  \widetilde{A}^\pvec(\rvec) =
  \begin{pmatrix}
    A^\pvec_{11} & A^\pvec_{12}\\
    A^\pvec_{21} & A^\pvec_{22}
  \end{pmatrix}
  :=
  \begin{pmatrix}
    \langle v_1, A(\rvec) v_1\rangle
    &
    \langle v_1, A(\rvec) v_2\rangle \\
    \langle v_2, A(\rvec) v_1\rangle
    &
    \langle v_2, A(\rvec) v_2\rangle
  \end{pmatrix}
\end{equation}
denote the submatrix of $A^\pvec$ corresponding to the eigenvectors of
the coalescing eigenvalues.  We stress again that the eigenvectors
$v_1=v_1(\pvec)$ and $v_2=v_2(\pvec)$ are computed at the point $\pvec$
and do not vary with $\rvec$.  By the definition of $A^\pvec$, we
have
\begin{equation}
  \label{eq:tildeA_diag}
  \widetilde{A}^\pvec(\pvec) =
  \begin{pmatrix}
    \lambda_1(\pvec) & 0\\
    0 & \lambda_2(\pvec)
  \end{pmatrix}.
\end{equation}

We let
\begin{equation}
  \label{eq:target_function_def}
  F\big(A^\pvec(\rvec)\big) := 
  \begin{pmatrix}
    A^{\pvec}_{11}(\rvec)-A^{\pvec}_{22}(\rvec)\\
    2A^{\pvec}_{12}(\rvec)
  \end{pmatrix}  
\end{equation}
denote the target function similar to \eqref{eq:F_def}.  We stress
that $F$ is a function of $\rvec$.

Throughout the paper $\Diff$ will denote the row vector of derivatives
taken with respect to parameters $\rvec = (x,y)$,
\begin{equation*}
  \Diff f = \left( \frac{\partial f}{\partial x},
    \ \frac{\partial f}{\partial y}\right).
\end{equation*}
If $f$ is a vector-function, $\Diff f$ is a matrix with 2 columns.  We
use the notation $\Diff_{\rvec_0} f$ to denote the derivative
evaluated at the point $\rvec=\rvec_0$, i.e.\
\begin{equation*}
  \Diff_{\rvec_0} f = \left( \frac{\partial f}{\partial x}(\rvec_0),
    \ \frac{\partial f}{\partial y}(\rvec_0)\right).
\end{equation*}

We use notation $J_{\rvec}(A^{\pvec})$ to denote the Jacobian of
$F(A^\pvec)$,
\begin{equation}
  \label{eq:Jacobian_def}
  J_{\rvec}(A^{\pvec}) :=
  \Diff_{\rvec} F(A^\pvec) =
  \begin{pmatrix} 
      \big\langle \vr_1, \spdev{x}{A}\vr_1\big\rangle
      -\big\langle \vr_2, \spdev{x}{A}\vr_2\big\rangle 
      & \big\langle \vr_1, \spdev{y}{A}\vr_1\big\rangle
      -\big\langle \vr_2, \spdev{y}{A}\vr_2\big\rangle \\
      2\big\langle \vr_1, \spdev{x}{A}\vr_2\big\rangle 
      & 2\big\langle \vr_1, \spdev{y}{A}\vr_2\big\rangle
    \end{pmatrix},
\end{equation}
where $v_1,v_2$ are the eigenvectors of $A(\pvec)$ and the derivatives
$\spdev{x}{A}$ and $\spdev{y}{A}$ have been evaluated at point
$\rvec$.  This is the matrix appearing in Theorem~\ref{thm:main}.  The
factor $2$ in the definition of $J_{\rvec}(A^{\pvec})$ arises
naturally in calculations; it can also be used to put the second row
terms in the more symmetric form,
\begin{equation*}
  2\Big\langle \vr_1, \spdev{x}{A}\vr_2\Big\rangle
  = \Big\langle \vr_1, \spdev{x}{A}\vr_2\Big\rangle
  + \Big\langle \vr_2, \spdev{x}{A}\vr_1\Big\rangle.  
\end{equation*}

Finally, we remark that by our definitions
$F(A) = F\big(\widetilde{A}\big)$ and
$J_\rvec(A) = J_\rvec\big(\widetilde{A}\big)$.  Therefore, the tilde
(defined in equation~\eqref{eq:tildeA_def}) will usually be omitted
once we invoke functions $F$ and $J$.

\section{Discussion}
\label{sec:discussion}

\subsection{Geometric interpretation}
\label{sec:berry_phase}

What is described in this paper is a variation of the Newton-Raphson
method searching for a zero of the objective function
$\lambda_1(\rvec)-\lambda_2(\rvec)$.  This is only one condition on
two parameters (in the real symmetric case), and leads to an underdetermined
Newton-Raphson iteration.  In particular, given an initial guess
$\rvec_0$, we would like to update our guess to $\rvec_1$ such that
\begin{equation}
  \label{eq:tangent_cond}
  \Diff_{\rvec_0}\big(\lambda_1(\rvec)-\lambda_2(\rvec)\big) \, (\rvec_1-\rvec_0)
  = -\big(\lambda_1(\rvec_0)-\lambda_2(\rvec_0)\big).
\end{equation}
However, there is a whole line of points $\rvec_1$ that satisfy this
condition, as illustrated in Figure~\ref{fig:tangent_planes}.
\begin{figure}
\centering
\subfloat[]{\includegraphics[width=0.24\textwidth]{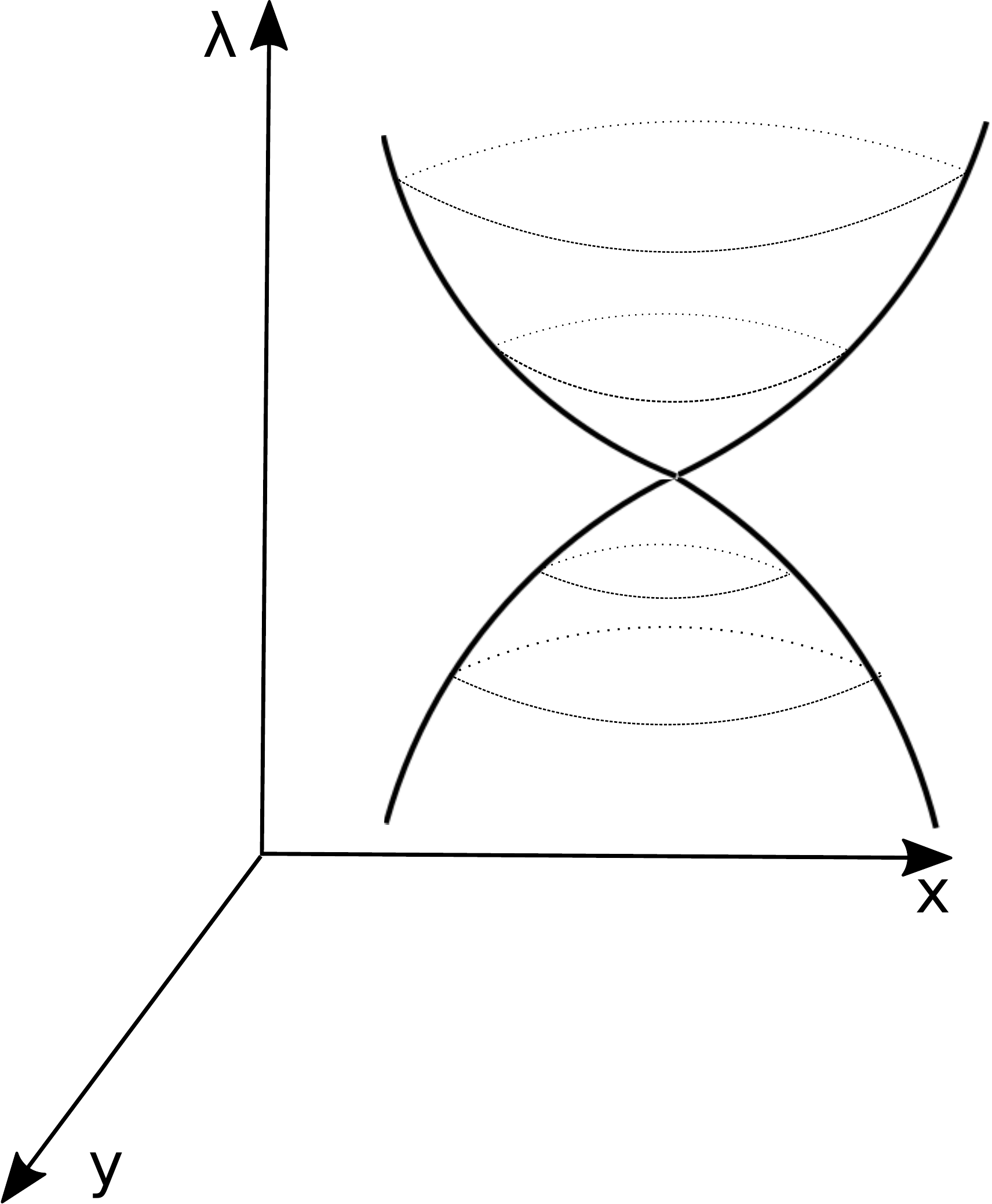}} 
\subfloat[]{\includegraphics[width=0.24\textwidth]{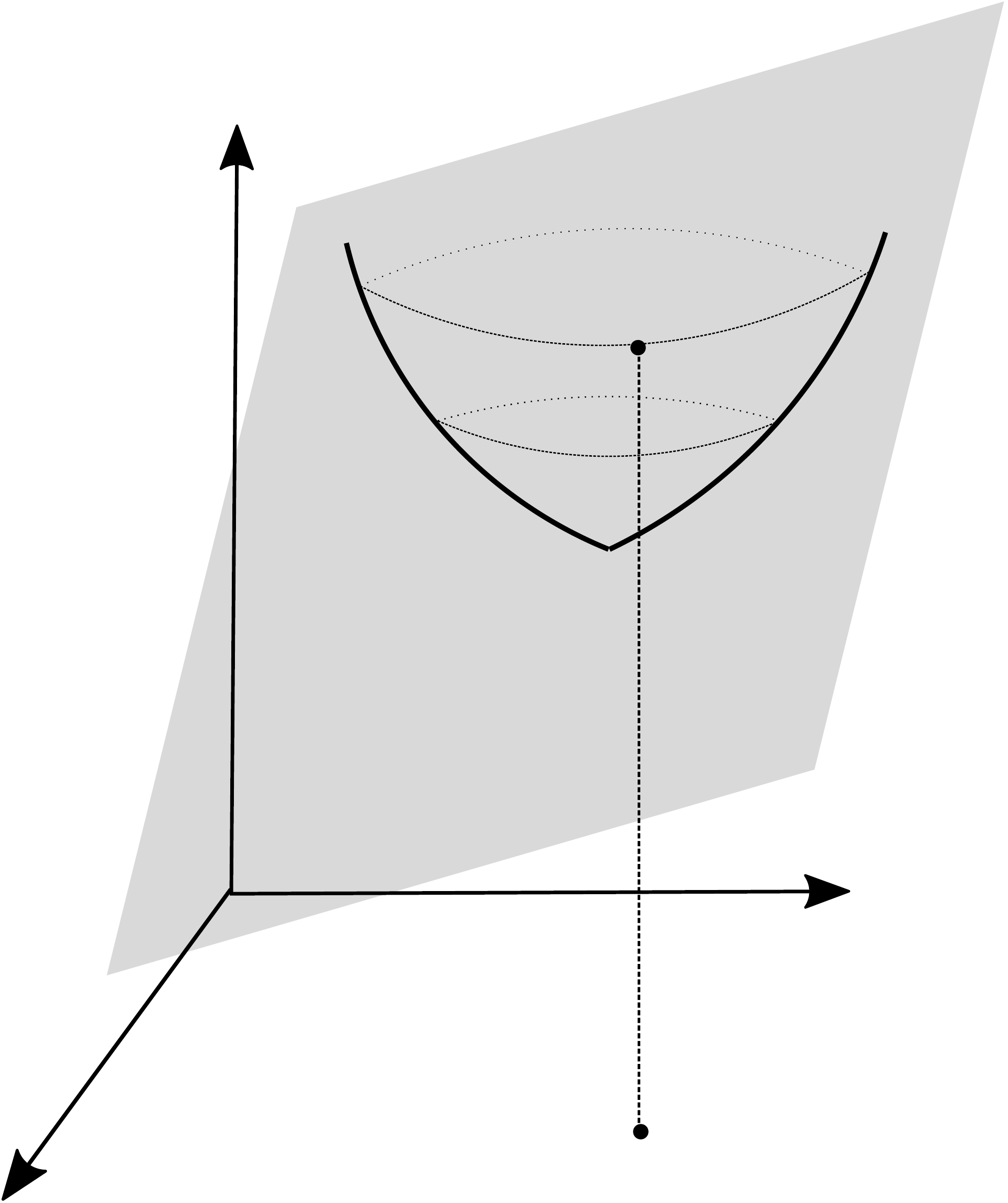}}
\subfloat[]{\includegraphics[width=0.24\textwidth]{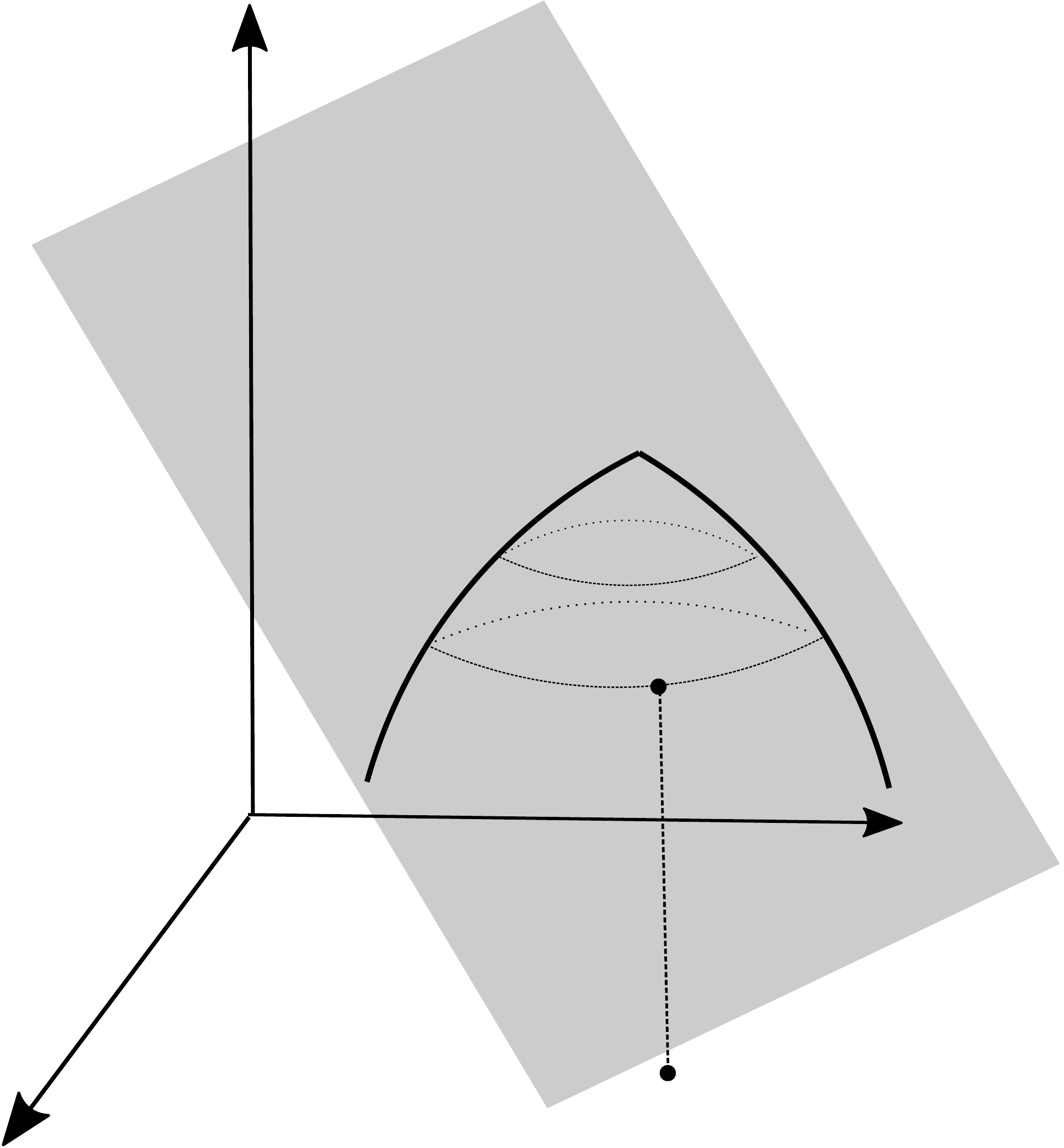}}
\subfloat[]{\includegraphics[width=0.24\textwidth]{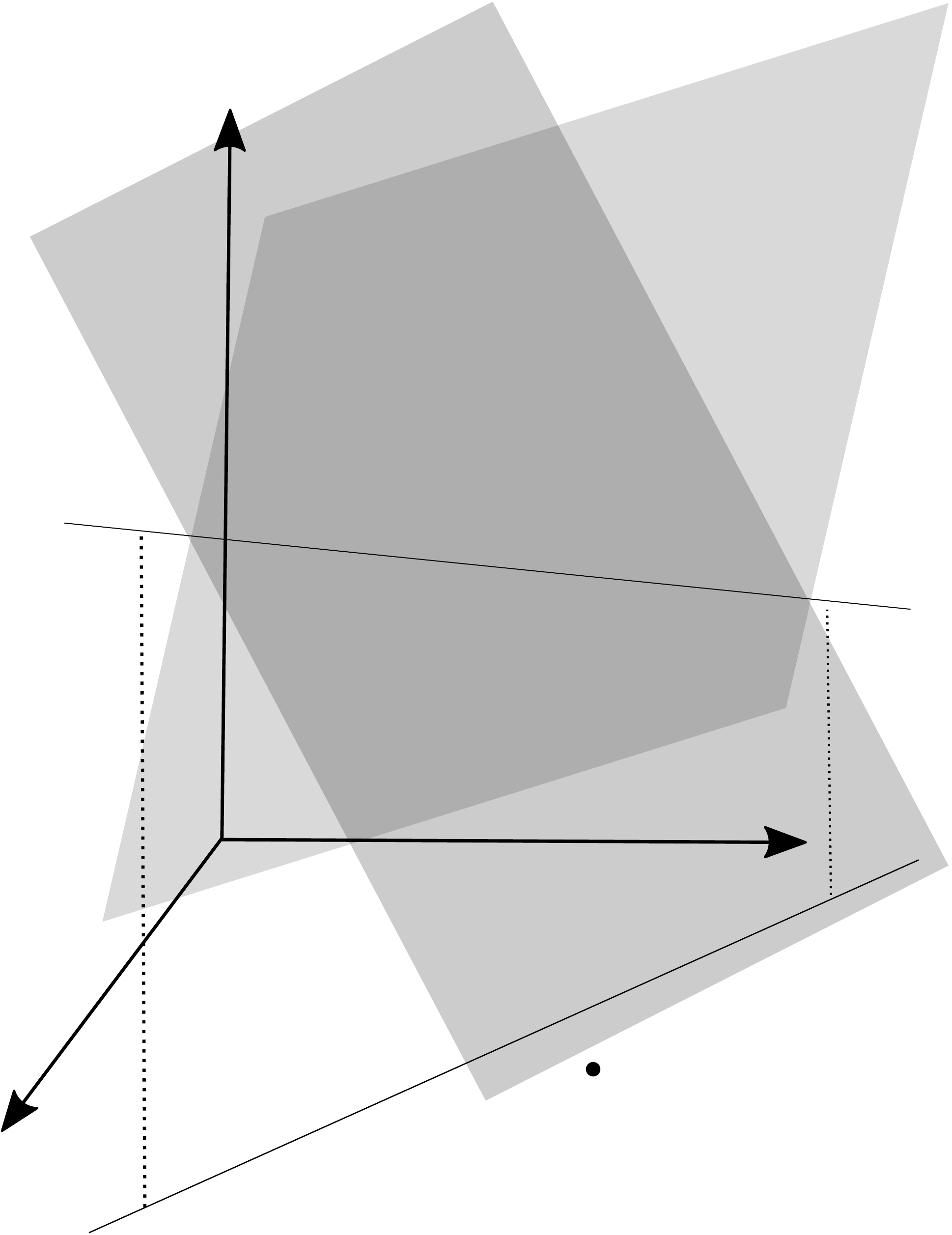}}
\caption{(a) Conical degeneracy of eigenvalues. (b) Linear
  approximation of top eigenvalue about the initial guess. (c) Linear
  approximation of bottom eigenvalue about the initial guess. (d) The
  intersection of the two linear approximations is a line, not a
  point. We need to use the conical nature of the intersection to
  determine a unique point to chose as our next guess.}
\label{fig:tangent_planes}
\end{figure}
To incorporate our knowledge that the degeneracy occurs at an isolated
point, we use a heuristic derived from Berry phase
\cite{HerLon_dfs63,Ber,Sim_prl83}, a phenomenon which underlies the
inability to find a smooth diagonalization around a degeneracy: on a
loop in the parameter space around a nondegenerate conical point, a
continuous choice of eigenvectors must rotate by $\pi$ (as opposed to
0 mod $2\pi$).

But if smoothly going in a loop around the degeneracy rotates the
eigenvectors, the direction of minimal rotation is a
direction \emph{towards the point of degeneracy}.  Let
$\{v_1(\rvec),v_2(\rvec)\}$ be a smooth choice of
normalized eigenvectors around the point $\rvec_0$ (this is possible
because $\rvec_0$ is not a point of eigenvalue multiplicity).  Then we
are looking for the direction in the parameter space in which the
eigenvector $v_1$ as a function of $\rvec$ does not
rotate in the plane spanned by $\{v_1(\rvec_0),v_2(\rvec_0)\}$
(it may still rotate ``out of the plane'').  This condition can be
written as
\begin{equation}\label{eq:nonrotation_cond}
  \Diff_{\rvec_0} \big\langle v_1(\rvec),
  v_2(\rvec_0)\big\rangle \, (\rvec_1-\rvec_0) = 0.
\end{equation}

Conditions~\eqref{eq:tangent_cond} and
\eqref{eq:nonrotation_cond} together generically\footnote{See
Sections~\ref{sec:conical_points} and \ref{sec:proofs} for a precise
formulation.} define a unique point $\rvec$ which can be taken as the
next step in the iteration.  We can solve for it explicitly using
the well-known perturbation formulas \cite{BorFoc,Kato_perturbation},
\begin{gather}
  \label{eq:pert_value}
  \Diff_{\rvec_0} \lambda_1 = \Diff_{\rvec_0} A^{\rvec_0}_{11},
  \qquad
  \Diff_{\rvec_0} \lambda_2 = \Diff_{\rvec_0} A^{\rvec_0}_{22},\\
  \label{eq:pert_vec}
  \Diff_{\rvec_0} \big\langle v_1(\rvec), v_2(\rvec_0)\big\rangle
   = \frac{\Diff_{\rvec_0} A^{\rvec_0}_{12}}{\lambda_1 - \lambda_2},      
\end{gather}
where
\begin{equation}
  \label{eq:Arij}
  A^{\rvec_0}_{ij}
  = A^{\rvec_0}_{ij}(\rvec)
  = \big\langle v_i(\rvec_0), A^{\rvec_0}(\rvec)v_j(\rvec_0) \big\rangle.
\end{equation}
We stress that in equation~\eqref{eq:Arij} the eigenvectors $v_1,v_2$
are evaluated at the point $\rvec_0$ and do not depend on $\rvec$.

The tangent planes condition~\eqref{eq:tangent_cond} and the
non-rotation condition~\eqref{eq:nonrotation_cond} can now be written
succinctly as
\begin{equation}
  \label{eq:two_cond_pert}
  \left[\Diff_{\rvec_0} 
    \begin{pmatrix}
      A^{\rvec_0}_{11}-A^{\rvec_0}_{22} \\
      2 A^{\rvec_0}_{12}
    \end{pmatrix}
    \right]
  (\rvec_1-\rvec_0)
  =
  \left[\Diff_{\rvec_0} F\big(A^{\rvec_0}(\rvec)\big) \right]
  (\rvec_1-\rvec_0)
  =
  \begin{pmatrix}
    \lambda_2-\lambda_1 \\
    0
  \end{pmatrix},
\end{equation}
or, less succinctly, as 
\begin{equation*}
  \begin{pmatrix} 
    \big\langle \vr_1, \spdev{x}{A}\vr_1\big\rangle
    -\big\langle \vr_2, \spdev{x}{A}\vr_2\big\rangle 
    & \big\langle \vr_1, \spdev{y}{A}\vr_1\big\rangle
    -\big\langle \vr_2, \spdev{y}{A}\vr_2\big\rangle \\
    2\big\langle \vr_1, \spdev{x}{A}\vr_2\big\rangle 
    & 2\big\langle \vr_1, \spdev{y}{A}\vr_2\big\rangle
  \end{pmatrix}
  (\rvec_1-\rvec_0)
  = 
  \begin{pmatrix}
    \lambda_2-\lambda_1 \\
    0
  \end{pmatrix},
\end{equation*}
which immediately leads to \eqref{iter_equation_simple}.

Berry phase also lies at the heart of another set of works devoted to
locating points of eigenvalue multiplicity.  Pugliese, Dieci and
co-authors
\cite{Pug_phd08,DiePug_mcs08,DiePug_siamjmaa09,DiePug_laa12,DiePapPug_siamjmaa13}
developed a procedure which uses Berry phase to grid-search available
space and identify regions with conical points.  For the final
convergence they used the standard Newton--Raphson method to locate
the critical point of $(\lambda_2-\lambda_1)^2$.  The convergence rate
of this final step is quadratic, as in Theorem~\ref{thm:main}; we
refer to Section~\ref{sec:mergingDirac} for a comparison of actual
convergence in an example.

In terms of ease of application, coding
equation~\eqref{iter_equation_simple} is straightforward and lack of
convergence of the method also carries information (see
Section~\ref{sec:avoided} and \ref{sec:mergingDirac}).  To perform a
thorough search of all available space and to locate all conical
points, it is preferable to use the methods of
\cite{Pug_phd08,DiePug_laa12,DiePapPug_siamjmaa13}.

\subsection{Relation to Newton-Raphson method}

Recalling the definition of $\widetilde{A}^{\rvec_0}$ and in
particular equation~\eqref{eq:tildeA_diag}, we have
\begin{equation*}
  \begin{pmatrix}
    \lambda_2-\lambda_1 \\
    0
  \end{pmatrix}
  = -F\big(A^{\rvec_0}(\rvec_0)\big).
\end{equation*}
This allows us to rewrite equation ~\eqref{eq:two_cond_pert} as
\begin{equation*}
  \left[ \Diff_{\rvec_0} F\big(A^{\rvec_0}(\rvec)\big) \right]
  \,(\rvec_1-\rvec_0)
  =  - F\big(A^{\rvec_0}(\rvec_0)\big),
\end{equation*}
which is the same as a single step of Newton--Raphson iteration applied
to $F(\widetilde{A}^{\rvec_0})$.  In other words,
$\rvec_1 = (x_1, y_1)$ is
chosen to be a solution to
\begin{equation}
  \label{first_order_expansion}
  \widetilde{A}^{\rvec_0}(\rvec_0)
  + (x_1-x_0)\spdev{x}{\widetilde{A}^{\rvec_0}}(\rvec_0)
  + (y_1-y_0)\spdev{y}{\widetilde{A}^{\rvec_0}}(\rvec_0)
  = \lambda I_2
\end{equation} 
for some $\lambda\in \mathbb{R}$.  Equivalently, $\rvec_1$ is the
point where the linear approximation to
$\widetilde{A}^{\rvec_0}(\rvec)$ has a double eigenvalue.

To understand the difference of our algorithm from a seemingly
conventional Newton--Raphson method, we need to revisit the
computation of $\widetilde{A}$.  It can be viewed as first expressing
$A(\rvec)$ in the eigenvector basis computed \emph{at the point} $\rvec_0$
and then extracting the $\{1,2\}$-sub-block of the resulting matrix.

In this notation, the problem of finding the degeneracy is equivalent
to finding a point $\rvec'$ such that
\begin{equation}
  \label{eq:wewant}
  \widetilde{A}^{\rvec'}(\rvec') = \lambda I_2,
  \qquad \mbox{for some }\lambda\in\R.
\end{equation}
In contrast, solving
equation~\eqref{first_order_expansion} is a first step in finding a
point $\rvec'$ such that
\begin{equation}
  \label{eq:wesolve}
  \widetilde{A}^{\rvec_0}(\rvec') = \lambda I_2,
  \qquad \mbox{for some }\lambda\in\R.
\end{equation}
Going all the way to find the solution $\rvec'$ to equation
\eqref{eq:wesolve} is pointless; this is not the equation we need to
solve.  Instead, we go one step, computing the first Newton--Raphson
approximation $\rvec_1$, and then update our target equation to
\begin{equation*}
  \widetilde{A}^{\rvec_1}(\rvec') = \lambda I_2,
  \qquad \mbox{for some }\lambda\in\R,
\end{equation*}
compute the first Newton--Raphson approximation $\rvec_2$ to
\emph{that} equation and so on.

\subsection{Complex Hermitian matrices}
\label{sec:complex}

Let us now consider a complex Hermitian matrix-valued function
$A\in C^2(\mathbb{R}^3, \mathbb{C}^{n\times n})$.  To find a point of
eigenvalue multiplicity, we typically need three real parameters (the
off diagonal terms can be complex, and that introduces an additional
degree of freedom), which we still denote by $\rvec = (x,y,z)$.

The conditions can now be written as
\begin{equation}
  \label{eq:three_cond_pert}
  \left[\Diff_{\rvec_0} G\big(A^{\rvec_0}(\rvec)\big)
  \right]
  (\rvec_1-\rvec_0)
  = 
  \begin{pmatrix}
    \lambda_2-\lambda_1 \\
    0 \\
    0
  \end{pmatrix},
\end{equation}
where
\begin{equation}
  \label{eq:objective_complex}
  G\big(A^{\rvec_0}\big) =
  \begin{pmatrix}
    A^{\rvec_0}_{11}-A^{\rvec_0}_{22} \\
    2 A^{\rvec_0}_{12} \\
    2 A^{\rvec_0}_{21}
  \end{pmatrix}.
\end{equation}
One can equivalently use the objective function
\begin{equation}
  \label{eq:objective_complex_ri}
  G\big(A^{\rvec_0}\big) =
  \begin{pmatrix}
    A^{\rvec_0}_{11}-A^{\rvec_0}_{22} \\
    2 \Real(A^{\rvec_0}_{12}) \\
    2 \Imag(A^{\rvec_0}_{21})
  \end{pmatrix}.
\end{equation}

\section{Conical Intersection}
\label{sec:conical_points}

Let $\alpha$ be a point in the parameter space such that $A(\alpha)$
has a double eigenvalue $\lambda_1=\lambda_2$.  The existence of
eigenvalue multiplicity precludes a smooth diagonalization in a region
containing the degeneracy. However, a smooth block diagonalization
exists.  The standard construction (see, for example, \cite[II.4.2 and
Remark 4.4 therein]{Kato_perturbation}) uses Riesz projector.

We can choose a contour $\gamma: [0,1]\mapsto \mathbb{C}$ with
$\gamma(0) = \gamma(1)$ enclosing $\lambda_1$, $\lambda_2$ and no other
point in the spectrum of $A(\alpha)$.  This property of $\gamma$ must
persist for $A(\rvec)$ when $\rvec$ is in a small neighborhood of
$\alpha$.  The Riesz projector
\begin{equation}
  \label{eq:Riesz_def}
  P(\rvec)
  = \int_{\gamma} (A(\rvec)-\lambda I_n)^{-1} d\lambda  
\end{equation}
projects onto the space spanned by the eigenvectors of
$\lambda_1(\rvec)$ and $\lambda_2(\rvec)$ \cite{HisSig}.  The
projector itself is smooth, as the points on the contour are all in
the resolvent set of $A$ (and so $A-\lambda I_n$ has a bounded inverse
for all $\lambda\in \Gamma$).  Starting with an arbitrary eigenvector
basis $\{v_1,v_2\}$ at $\alpha$, we can obtain a basis at a nearby
$\rvec$ by applying Gram-Schmidt procedure to the set
$\left\{P(\rvec) v_1, P(\rvec) v_2\right\}$, which preserves
smoothness.  We can do the same with the orthogonal complement
$I-P(\rvec)$ and a complementary basis to $\{v_1,v_2\}$.  To
summarize, for some region $\Omega\in \mathbb{R}^2$ with
$\alpha\in \Omega$, we find a change of basis
$M(\cdot)\in C^2(\Omega, R^{n\times n})$ such that
\begin{equation}
  \label{block_diagonalization}
  M(\rvec)^*A(\rvec)M(\rvec) = B(\rvec)\oplus\Lambda(\rvec),
\end{equation}
where $B\in C^2(\Omega, \mathbb{R}^{2\times 2})$ and
$\Lambda\in C^2(\Omega, \mathbb{R}^{(n-2)\times (n-2)})$.
We can further diagonalize both $B$ and $A$ at any point $\rvec_0$ to
obtain 
\begin{equation}
  \label{diag_block_diagonalization}
  \Gamma(\rvec)^*A^{\rvec_0}(\rvec)\Gamma(\rvec)
  = B^{\rvec_0}(\rvec)\oplus\Lambda(\rvec),
\end{equation}
where $\Gamma(\rvec) = V^TM(\rvec)(W\oplus I_{n-2})
\in C^2(\Omega, R^{n\times n})$, and both
\begin{equation*}
  A^{\rvec_0}(\cdot) := V^TA(\cdot)V
  \qquad\mbox{and}\qquad 
  B^{\rvec_0}(\cdot) := W^TB(\cdot)W  
\end{equation*}
are diagonal at $\rvec_0$.
A stronger result from Hsieh, and Sibuya \cite{HsiSib}, and Gingold
\cite{Gin} states that such block-diagonalization exists even for
matrices that are not necessarily Hermitian, and for any closed
rectangular region that contains an isolated degeneracy.

Note that since $B$ is a $2\times 2$ matrix which has an eigenvalue
multiplicity at the point $\alpha$, $B(\alpha)$ is a multiple of the
identity.  The eigenvalue multiplicity is detected by the
\emph{discriminant} of $B$ which in the $2\times2$ case is defined as
\begin{equation}
  \label{eq:discr_B}
  \disc(B) := (\lambda_1-\lambda_2)^2 = (B_{11}-B_{22})^2+4B_{12}^2.
\end{equation}
The discriminant achieves its minimum value 0 at the point $\alpha$.
It is also a $C^2$ function of $\rvec$ and its Hessian is well-defined.

\begin{definition}\label{def_conical_degeneracy}
  A point of eigenvalue multiplicity $\alpha$ is \emph{a non-degenerate
    conical point} if $\disc(B(\rvec))$ has a non-degenerate critical
  point at $\rvec=\alpha$.
\end{definition}

In other words, there is a positive definite matrix $H$ (the
``Hessian'') such that
\begin{equation*}
  \disc(B(\rvec))
  = \big\langle (\rvec-\alpha), H(\rvec-\alpha)\big\rangle
  + o\!\left(|\rvec-\alpha|^2\right),
\end{equation*}
and therefore, along any ray originating at $\alpha$, the eigenvalues
are separating at a non-zero linear rate.  This picture justifies the
use of the term ``conical''.

Unfortunately, while existence of $B(\rvec)$ is assured, it is not
easily accessible.  The following theorem provides a more
practical method of checking if $\alpha$ is non-degenerate.

\begin{theorem}
  \label{thm:Hess_disc}
  The Hessian of $\disc(B)$ at $\alpha$ is given by
  \begin{equation}
    \label{eq:Hess_disc}
    \Hess_\alpha (\disc(B))
    = 2 J_{\alpha}(B)^T J_{\alpha}(B)
    = 2 J_\alpha(A^\alpha)^T J_\alpha(A^\alpha).
  \end{equation}
  Consequently, $\alpha$ is a non-degenerate conical point if and only
  if $\det J_\alpha(A^\alpha) \neq 0$.
\end{theorem}

The condition $\det J_\alpha(A^\alpha) \neq 0$ has a nice geometric
meaning: it is precisely the condition that the manifold
$\widetilde{A}^\alpha$ of $2\times2$ real symmetric matrices is
transversal to the line of $2\times2$ symmetric matrices with repeated
eigenvalues (cf.\ \cite[Def.\ 1]{ONe_siamjmaa05}).

The choice of basis in the definition of
$\widetilde{A}^\alpha$ is assumed to align with the choice of basis
used to compute $B(\rvec)$, i.e.\ the first two columns of $M(\alpha)$
are the eigenvectors used to compute $\widetilde{A}^\alpha$.  This
choice does not affect the definition of the non-degenerate point
because of the following lemma.

\begin{lemma}
  \label{lem:basis_change_FJ}
  Let $A\in C^2(\mathbb{R}^2, \mathbb{R}^{2\times 2})$ be a $2\times2$
  matrix-valued function of $\rvec\in \mathbb{R}^2$.  Then for any 
  orthogonal matrix $U\in\mathbb{R}^{2\times 2}$ there is an
  orthogonal matrix $W\in\mathbb{R}^{2\times 2}$ such that for
  all
  $\rvec$ we
  have
  \begin{equation}
    \label{eq:conj_FJ}
    F(U^TAU) = WF(A), \qquad
    J_\rvec(U^TAU) = W J_\rvec(A),    
  \end{equation}
  and therefore
  \begin{equation}
    \left|\det(J_\rvec(A))\right| =
    \left|\det(J_\rvec(U^TAU))\right|.
  \end{equation}
\end{lemma}

\begin{proof}
  This identity for $2\times2$ matrix-functions can be checked by
  direct computation but the details are excessively tedious.  Instead
  we use a more generalizable approach.

  We fix an orthogonal $U$ and let $\mathcal{S}^2$ denote the linear
  space of $2\times2$ real symmetric matrices.  The map $F$, see
  equation~\eqref{eq:target_function_def}, acts as a linear
  transformation from $\mathcal{S}^2$ to $\mathbb{R}^2$.  It is
  obviously onto and has the kernel $\Ker(F)$ consisting of multiples 
  of the identity. On the other hand, conjugation by $U$ (namely the map
  $A \mapsto U^TAU$) is a linear transformation of $\mathcal{S}^2$ to
  itself.  It maps multiples of the identity to themselves and
  therefore induces a linear transformation from the quotient space
  $\mathcal{S}^2 / \Ker(F)$ to itself.  This linear transformation,
  via the isomorphism $F$ between $\mathcal{S}^2 / \Ker(F)$ and
  $\mathbb{R}^2$, induces a linear transformation on $\mathbb{R}^2$
  mapping $F(A)$ to $F(U^TAU)$.
  
  We summarize the above in the
  commutative diagram
  \begin{equation*}
    \begin{CD}
      \mathcal{S}^2 @>{A \mapsto U^TAU}>> \mathcal{S}^2 \\
      @V{F}VV @V{F}VV\\
      \mathbb{R}^2 @>W>> \mathbb{R}^2
    \end{CD}
  \end{equation*}
  In other words, for a given orthogonal $U$, there exists a constant
  $2\times2$ matrix $W$ such that
  \begin{equation*}
    F(U^TAU) = WF(A).
  \end{equation*}
  From the identity (see \eqref{eq:discr_B} for the definition of
  discriminant)
  \begin{equation*}
    \left|F(A)\right|^2
    = \disc(A)
    = \disc(U^TAU)
    = \left|F(U^TAU)\right|^2    
  \end{equation*}
  we conclude that $W$ is orthogonal.
  Finally, taking derivatives we get
  \begin{equation*}
    J(U^TAU) = W J(A),      
    \quad\implies\quad
    \det(J(U^TAU)) = \det(WJ(A)) = \pm\det(J(A)),  
  \end{equation*}
  since determinant of an orthogonal matrix is either $1$ or $-1$.
\end{proof}

The following identity will be helpful in the proof of
Theorem~\ref{thm:Hess_disc} and also in Section~\ref{sec:proofs}.

\begin{lemma}
  \label{lem:jacobian_difference}
  For any $A^{\rvec_0}$ and $B^{\rvec_0}$ as in
  equation~\eqref{diag_block_diagonalization}, 
  \begin{equation}
    \label{jacobian_difference}
    J_{\rvec_0}(B^{\rvec_0})
    = J_{\rvec_0}(A^{\rvec_0})
    + 2(\lambda_2-\lambda_1)
    \begin{pmatrix}
      0& 0\\
      \big\langle \spdev{x}{\gamma_1}, \gamma_2\big\rangle
      & \big\langle \spdev{y}{\gamma_1}, \gamma_2\big\rangle
    \end{pmatrix},
  \end{equation}
  where $\gamma_{1,2}=\gamma_{1,2}(\rvec_0)$ are the first two columns
  of the matrix $\Gamma(\rvec_0)$.
\end{lemma}

\begin{proof}
  We remark that identity~\eqref{jacobian_difference} is only claimed
  for the Jacobian evaluated at the point where both $A^{\rvec_0}$ and
  $B^{\rvec_0}$ are diagonal, therefore $A^{\rvec_0}\gamma_j(\rvec_0)
  = \lambda_j(\rvec_0)\gamma_j(\rvec_0)$.

  For all $\rvec$, $\gamma_j(\rvec)$ are orthonormal and
  differentiating
  $\langle \gamma_i, \gamma_j \rangle = \mathrm{const}$ we get
  \begin{equation}
    \label{eq:commutation_gamma}
    \left\langle\spdev{x}{\gamma_i}, \gamma_j\right\rangle
    = -\left\langle\gamma_i, \spdev{x}{\gamma_j}\right\rangle.
  \end{equation}

  We can now relate the derivatives of
  $A^{\rvec_0}$ to the derivatives of $B^{\rvec_0}$,
 \begin{align*}
   \pdev{x}{B^{\rvec_0}_{ij}}
   &=\frac{\partial}{\partial x} \big\langle \gamma_j,
     A^{\rvec_0}\gamma_i\big\rangle\\
   &=\left\langle \gamma_i, \spdev{x}{A^{\rvec_0}} \gamma_j\right\rangle
     +\left\langle\spdev{x}{\gamma_i}, A^{\rvec_0}\gamma_j\right\rangle
     +\left\langle \gamma_i, A^{\rvec_0}\spdev{x}{\gamma_j}\right\rangle\\
   &=\spdev{x}{A^{\rvec_0}_{ij}}
     + \lambda_j \left\langle\spdev{x}{\gamma_i}, \gamma_j\right\rangle
     + \lambda_i \left\langle\gamma_i, \spdev{x}{\gamma_j}\right\rangle\\
   &=\spdev{x}{A^{\rvec_0}_{ij}}
     + (\lambda_j- \lambda_i)
     \left\langle\spdev{x}{\gamma_i}, \gamma_j\right\rangle,
     \qquad i,j \in \{1,2\}.
 \end{align*}
 The calculation is identical for $y$ derivatives.
\end{proof}

\begin{proof}[Proof of Theorem~\ref{thm:Hess_disc}]
  We write
  \begin{equation*}
    \disc(B) = (B_{11}-B_{22})^2+4B_{12}^2
    = \left\langle F(B), F(B) \right\rangle,
  \end{equation*}
  and note that $F(B(\alpha)) = \mathbf{0}$.  The latter observation
  implies that the product rule for the second derivatives at the
  point $\alpha$ collapses to
  \begin{equation*}
    \frac{\partial^2}{\partial x_i \partial x_j}
    \left\langle F(B), F(B) \right\rangle
    = 2 \left\langle \frac{\partial F(B)}{\partial x_i},
      \frac{\partial F(B)}{\partial x_j} \right\rangle,
    \qquad x_i,x_j \in \{x,y\}.
  \end{equation*}
  Therefore the Hessian can be written as
  \begin{equation*}
    \Hess_\alpha \langle F(B), F(B) \rangle =
    2
    \begin{pmatrix}
      \frac{\partial F(B)^T}{\partial x}\\[5pt]
      \frac{\partial F(B)^T}{\partial y}      
    \end{pmatrix}
    \begin{bmatrix}
      \frac{\partial F(B)}{\partial x} 
      &\frac{\partial F(B)}{\partial y}
    \end{bmatrix} = 2J_{\alpha}(B)^TJ_{\alpha}(B).
  \end{equation*}
  Finally, setting $\rvec_0=\alpha$ in
  Lemma~\ref{lem:jacobian_difference} yields  
  \begin{equation}
    \label{eq:deriv_equal}
    J_{\alpha}(B) = J_{\alpha}(A^\alpha),
  \end{equation}
  and concludes the proof of \eqref{eq:Hess_disc}.
 \end{proof}

\section{Proof of the main result}
\label{sec:proofs}

Here we restate the procedure used to locate the degeneracy in the
notation that has been introduced.

\begin{theorem}
  \label{main_thm_proofs}
  Let $\sigma\colon C^2(\mathbb{R}^2, \mathbb{R}^{2\times 2})
  \times \mathbb{R}^2\to \mathbb{R}^2$ be defined by
  \begin{equation}
    \label{eq:sigma_def}
    \sigma(S, \rvec) = \rvec - J_{\rvec}(S)^{-1}F_\rvec(S).    
  \end{equation}
  
  Let $A\in C^2(\mathbb{R}^2, \mathbb{R}^{n\times n})$ have a
  non-degenerate conical point at $\alpha$ between eigenvalues
  $\lambda_1$ and $\lambda_2$. Then there exists an open
  $\Omega\subset \mathbb{R}^2$ with $\alpha\in \Omega$ and
  $\exists C\in \mathbb{R}$, such that for all $\rvec\in \Omega$,
  \begin{equation}
    \label{eq:sigma_est}
    |\sigma(\widetilde{A}^{\rvec}, \rvec)-\alpha|
    < C|\rvec-\alpha|^2,    
  \end{equation}
  where the $2\times2$ matrix-function
  $\widetilde{A}^{\rvec}(\cdot)\in C^2(\mathbb{R}^2,
  \mathbb{R}^{2\times 2})$ is defined by
  \begin{equation}
    \label{eq:Atilde_def}
    \widetilde{A}^{\rvec}(\cdot) = V^T A(\cdot) V,
  \end{equation}
  with the constant $n\times 2$ matrix $V = (v_1\ v_2)$ whose columns
  are the eigenvectors of $A(\rvec)$.
\end{theorem}

We remark that non-degeneracy of the conical point is a generic
property: any degenerate conical point can be made non-degenerate by a
small perturbation of the function $A$.

We recall that the superscript in
$\widetilde{A}^{\rvec}(\cdot)$ refers to the basis which is computed at the
point $\rvec$ and in which the matrix $A(x,y)$ is represented.  
The derivatives of $\widetilde{A}^{\rvec}(\cdot)$ that are taken to
compute $J_\rvec$ in \eqref{eq:sigma_def}, are also evaluated at the
point $\rvec$.  The result of evaluating
$\sigma(\widetilde{A}^{\rvec},\rvec)$ is explicitly written out in
equation~\eqref{iter_equation_simple}.

\begin{proof}
  Let $B$ be the matrix defined in equation
  (\ref{block_diagonalization}). We will see, in Lemmas \ref{sigma_B}
  and \ref{sigma_diff} below, that there is a neighborhood
  $\Omega\subset \mathbb{R}^2$ of the conical point $\alpha$, and constants
  $C_1, C_2 > 0$ such that for all $\rvec\in \Omega$ we have
  \begin{equation*}
    |\sigma(B, \rvec)-\alpha|<C_1|\rvec-\alpha|^2  
  \end{equation*}
  and
  \begin{equation*}
    |\sigma(B,\rvec)-\sigma(\widetilde{A}^{\rvec},\rvec)|<C_2|\rvec-\alpha|^2.  
  \end{equation*}
  Together, these give us
  \begin{equation*}
    |\sigma(\widetilde{A}^{\rvec},\rvec)-\alpha|<(C_1+C_2)|\rvec-\alpha|^2,    
  \end{equation*}
  as desired.
\end{proof}

Now we establish the lemmas used in the proof of
Theorem~\ref{main_thm_proofs}.

\begin{lemma}
  \label{sigma_B}
  There exists $\Omega_1\subset \mathbb{R}^2$ with $\alpha\in \Omega_1$ and
  $C_1\in \mathbb{R}$ such that
  \begin{equation}
    \label{eq:newton_B}
    |\sigma(B,\rvec)-\alpha|<C_1|\rvec-\alpha|^2,  
  \end{equation}
  when $\rvec\in \Omega_1$.
\end{lemma}

\begin{proof}
  This is the usual Newton--Raphson method applied to conical point
  search for the $2\times2$ matrix $B$.  For completeness we provide
  the proof.  For the function $F(\rvec) := F(B(\rvec))$, we have the
  Taylor expansion around the point $\rvec_0$ which is evaluated at
  the point $\alpha$,
  \begin{equation*}
    \textbf{0} = F(\alpha) = F(\rvec_0)
    + \Diff_{\rvec_0} F \cdot (\alpha-\rvec_0)
    + O(|\alpha-\rvec_0|^2),
  \end{equation*}
  where the constant in $O(|\alpha-\rvec_0|^2)$ is \emph{independent}
  of $\rvec_0$ as long as it is in a neighborhood
  $\widetilde{\Omega}_1$ of $\alpha$.  The dot denotes the
  matrix-by-vector multiplication (to distinguish it from the argument
  of the function $F$).

  By assumption $\det(J_{\alpha}) \ne 0$, and, by smoothness, we know
  that $\Diff_{\rvec_0} F = J_{\rvec_0}$ is boundedly invertible in some region
  $\Omega_1 \subset \widetilde{\Omega}_1$ containing $\alpha$.
  Therefore, for the point $\rvec_1 = \sigma(B,\rvec_0)$, or equivalently,
  \begin{equation*}
    J_{\rvec_0} \cdot (\rvec_1-\rvec_0) = -F(\rvec_0),    
  \end{equation*}
  we have
  \begin{equation*}
    \textbf{0} 
    = J_{\rvec_0}\cdot (\alpha-\rvec_1) + O(|\alpha-\rvec_0|^2),
  \end{equation*}
  with the estimate \eqref{eq:newton_B} following by inverting
  $J_{\rvec_0}$.
\end{proof}

\begin{lemma}
  \label{invariance_of_iteration}
  For any $B\in C^2(\mathbb{R}^2, \mathbb{R}^{n\times n})$ and
  constant, orthogonal $U$, we have
  \begin{equation}
    \label{eq:invariance_of_iteration}
    \sigma(B, \rvec)=\sigma(U^TBU, \rvec).    
  \end{equation}
\end{lemma}

\begin{proof}
  Equation~\eqref{eq:invariance_of_iteration} follows directly from
  the definition of the one-step iteration function $\sigma$ and
  Lemma~\ref{lem:basis_change_FJ}.
\end{proof}

\begin{lemma}
  \label{sigma_diff}
  There exists $\Omega_2\subset \mathbb{R}^2$ with $\alpha\in \Omega_2$ and
  $C_2\in \mathbb{R}$ such that
  \begin{equation}
    \label{eq:sigma_diff}
    |\sigma(B,\rvec)
    - \sigma(\widetilde{A}^{\rvec},\rvec)|
    < C_2|\rvec-\alpha|^2,
  \end{equation}
  when $\rvec\in \Omega_2$.
\end{lemma}

\begin{proof}
  By the assumption that $\alpha$ is a non-degenerate conical point
  and equation~\eqref{eq:Hess_disc}, we have that $J_\rvec(B)$ and
  therefore $J_\rvec(B^{\rvec})$ has a bounded inverse in a region
  around $\alpha$.  By equation \eqref{jacobian_difference} we
  conclude that $J_\rvec(\widetilde{A}^{\rvec})$ also has a bounded
  inverse in some region $\Omega_2$ around $\alpha$ where
  $\lambda_1-\lambda_2$ is small. We can express the difference of the
  inverses as
  \begin{align*}
    J_\rvec(B^{\rvec})^{-1}-J_\rvec(\widetilde{A}^{\rvec})^{-1}
    &= J_\rvec(B^{\rvec})^{-1}
      \left(J_\rvec(\widetilde{A}^{\rvec}) -
      J_\rvec(B^{\rvec})\right)
      J_\rvec(\widetilde{A}^{\rvec})^{-1} \\
    &= (\lambda_1-\lambda_2)
      J_\rvec(B^{\rvec})^{-1}
      \begin{pmatrix}
        0 & 0\\
        \big\langle \spdev{x}{\gamma_1},
        \gamma_2\big\rangle & \big\langle
        \spdev{y}{\gamma_1}, \gamma_2\big\rangle
      \end{pmatrix}
          J_\rvec(\widetilde{A}^{\rvec})^{-1}.      
  \end{align*}
  and so, using boundedness of $\Gamma$ and its derivatives, we get
  \begin{equation*}
    \left\|J_\rvec(B^{\rvec})^{-1}
      - J_\rvec(\widetilde{A}^{\rvec})^{-1}\right\|
    = O(\lambda_1-\lambda_2).
  \end{equation*}
  We also recall that by definition of $A^\rvec$ and $B^\rvec$, 
  \begin{equation*}
    F(B^\rvec) = 
    F(\widetilde{A}^{\rvec}) =
    \begin{pmatrix}
      \lambda_1(\rvec) - \lambda_2(\rvec) \\
      0
    \end{pmatrix}.
  \end{equation*}
  
  Finally, abbreviating $J=J_{\rvec}$, we estimate
  \begin{align*}
    \left|\sigma(B^{\rvec},\rvec)
    -\sigma(\widetilde{A}^{\rvec},\rvec)\right|
    &= \left|J(B^{\rvec})^{-1}F(B^{\rvec})
      -J(\widetilde{A}^{\rvec})^{-1}F(\widetilde{A}^{\rvec})\right|\\
    &= \left|\left(J(B^{\rvec})^{-1}
      -J(\widetilde{A}^{\rvec})^{-1}\right)F(\widetilde{A}^{\rvec})\right|\\
    &\le \left\|J(B^{\rvec})^{-1}
      -J(\widetilde{A}^{\rvec})^{-1}\right\|\,
      \left|F(\widetilde{A}^{\rvec})\right|\\
    &= O\left((\lambda_2-\lambda_1)^2\right)
      = O\left(|\rvec-\alpha|^2\right).
  \end{align*}
  Equation~\eqref{eq:sigma_diff} now follows by applying
  Lemma~\ref{invariance_of_iteration} to get
  $\sigma(B^{\rvec},\rvec) = \sigma(B,\rvec)$.
\end{proof}

\section{Examples}
\label{sec:examples}

\subsection{Elements of A are linear in parameters}
If $A$ is linear in each parameter, we have $A = \Lambda I + xA_x +
yA_y = \Lambda I + \alpha I+\beta \sigma_1 + \gamma \sigma_3$, where
\begin{equation*}
  \sigma_1 = \begin{pmatrix}0&1 \\1&0\end{pmatrix}
  \qquad\mbox{and}\qquad
  \sigma_3 = \begin{pmatrix}1&0 \\0&-1\end{pmatrix},
\end{equation*}
for some $\alpha$, $\beta$ that depend on $x, y$, and $A$. The eigenvalues of this matrix are values of $\lambda$ where
\[\det(A-\lambda I) = \det(\Lambda I + \alpha I+\beta \sigma_1 +
  \gamma \sigma_3-\lambda I) = 0\]
\[(\Lambda+\alpha-\lambda)^2 = \beta^2+\gamma^2\]  
\[\lambda = \Lambda + \alpha \pm \sqrt{\beta^2+\gamma^2}\]
which is a cone in the new parameter space. In fact, a simple calculation shows that the degeneracy of the function $\hat{A}(\alpha, \beta) = \begin{pmatrix} \beta & \gamma \\ \gamma & -\beta\end{pmatrix}$, which has the same eigenvectors and shifted eigenvalues, can be located using a single step of iteration~\eqref{iter_equation_simple}.

\subsection{Non-linear examples}

Consider the following matrix-function example,
\begin{equation}
  \label{eq:real_convergence}
  A(x,y) = 
  \begin{pmatrix}
    2\cos(x)&0&0&1\\
    0&0.5+\cos(y)&0&1\\
    0&0&1&1\\
    1&1&1&1
  \end{pmatrix}.  
\end{equation}
Since $A(x,y)$ is a rank-one perturbation of a diagonal matrix, it can
be shown that there is a double eigenvalue $1$ at the point given by
\begin{equation*}
  2\cos(x) = 0.5 + \cos(y) = 1,
\end{equation*}
or $x=y=\pi/3$.  The results of running the algorithm of
Theorem~\ref{thm:main} with random starting points in the rectangle
$(\frac\pi3, \frac\pi3) \pm \frac12$ is shown in 
Figure~\ref{fig:real_convergence}.

\begin{figure}
  \centering
  \subfloat[]{\label{fig:real_convergence}
    \includegraphics[width=0.4\textwidth]{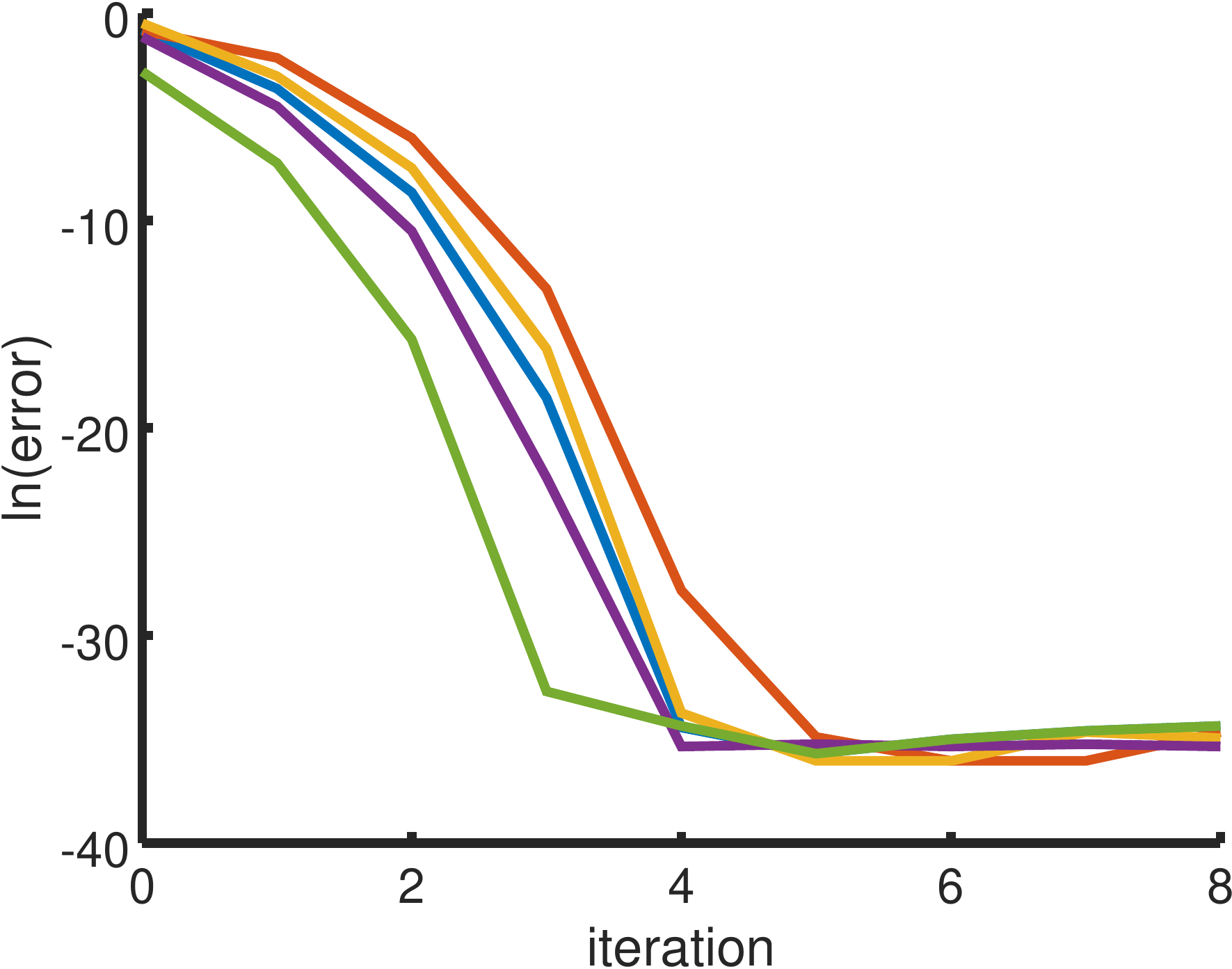}} 
  \subfloat[]{\label{fig:complex_convergence}
    \includegraphics[width=0.4\textwidth]{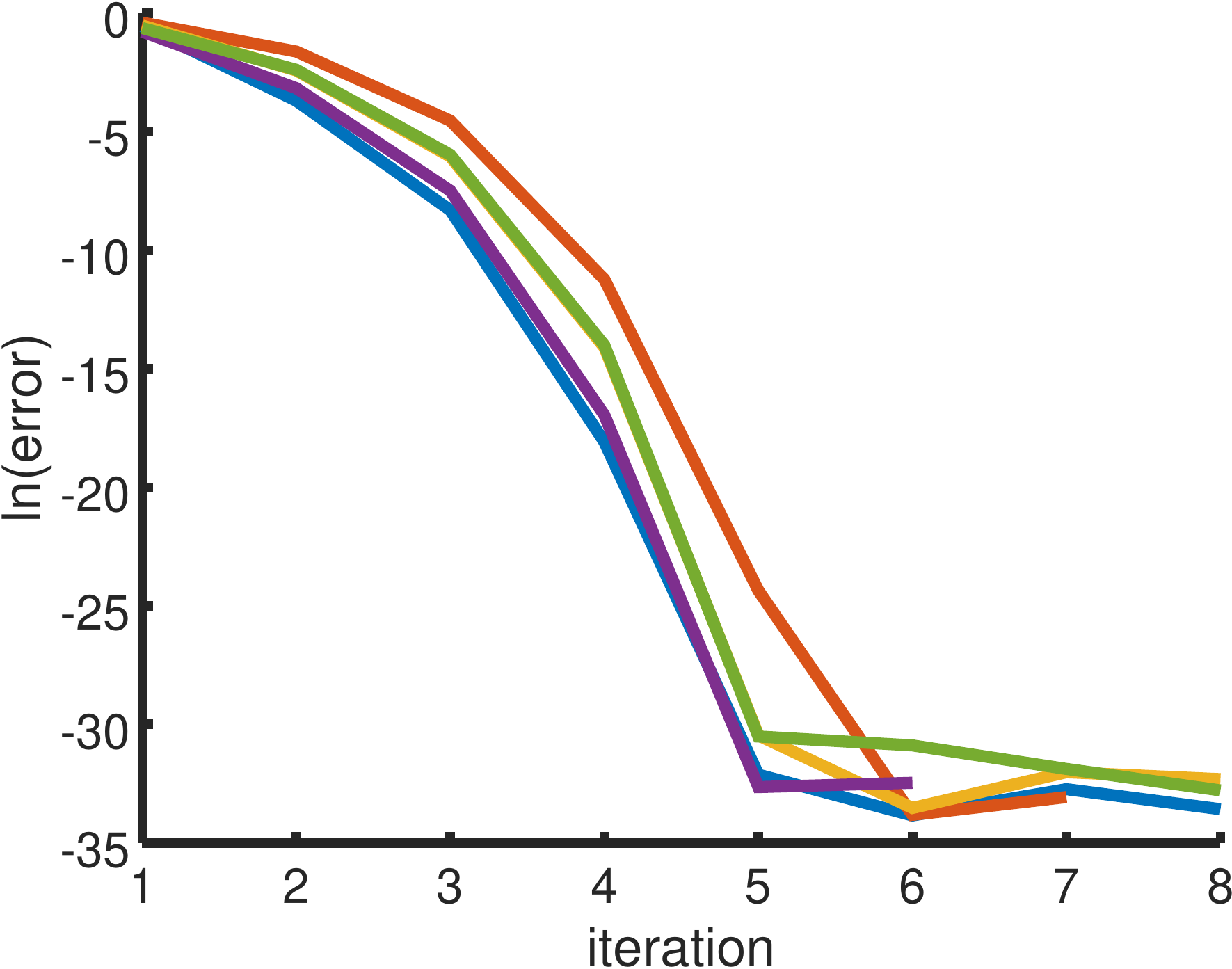}}\\
  \caption{(A) Logarithm of the distance from the $i$-th iteration
    $\rvec_i$ to the conical point $(\frac\pi3,\frac\pi3)$ of $A(x,y)$
    from equation (\ref{eq:real_convergence}), plotted as a function
    of $i$; the algorithm saturates at the limit of numerical
    precision in 3-5 steps.  (B) Logarithm of
    $|\rvec_{i+1}-\rvec_{i}|$ where $\rvec_i$ is the $i$-th iteration
    of the algorithm applied to $A(x, y, z)$ given by equation
    (\ref{eq:complex_convergence}). Several independent runs are
    plotted, each beginning at a random point in $[-\pi, \pi]$.}
\end{figure}

The complex Hermitian
case described in Section \ref{sec:complex} is demonstrated in Figure
\ref{fig:complex_convergence}.  The matrix
\begin{equation}
  \label{eq:complex_convergence}
  A =
  \begin{pmatrix}
    1&1&0&0&0&0&0&0&0&z\\
    1&3&e^{ix}&1&0&0&0&0&0&0\\
    0&e^{-ix}&2&1&0&0&0&0&0&0\\
    0&1&1&3&1&0&0&0&0&0\\
    0&0&0&1&3&1&1&0&0&0\\
    0&0&0&0&1&3&0&0&0&0\\
    0&0&0&0&1&0&3&1&1&0\\
    0&0&0&0&0&0&1&2&e^{iy}&0\\
    0&0&0&0&0&0&1&e^{-iy}&3&1\\
    z&0&0&0&0&0&0&0&1&1
  \end{pmatrix}.  
\end{equation}
corresponds to the discrete Laplacian of the graph shown in
Figure~\ref{fig:graph_Laplacian} with dashed edges carrying a magnetic
potential ($x$ and $y$ correspondingly).  The parameter $z$ is
introduced artificially, and the conical point found numerically has
value $z=0$.  Since the location of the conical point is not known
analytically, the error is estimated using the norms of updates
$\left\|\rvec_{i}-\rvec_{i+1}\right\|$ instead of
$\left\|\rvec_{i}-\alpha\right\|$.  The result of several runs of the
algorithm is shown in Figure~\ref{fig:complex_convergence}.

\begin{figure}
\centering
\begin{tikzpicture}
\draw[fill=black] (0,0) circle (3pt);
\draw[fill=black] (-2,0) circle (3pt);
\draw[fill=black] (2,0) circle (3pt);
\draw[fill=black] (0,2) circle (3pt);
\draw[fill=black] (4,0) circle (3pt);
\draw[fill=black] (4,2) circle (3pt);
\draw[fill=black] (6,0) circle (3pt);
\draw[fill=black] (8,2) circle (3pt);
\draw[fill=black] (8,0) circle (3pt);
\draw[fill=black] (10,0) circle (3pt);
\node at (0,-0.5) {2};
\node at (-2,-0.5) {1};
\node at (2,-0.5) {4};
\node at (0,2.5) {3};
\node at (4,-0.5) {5};
\node at (4,2.5) {6};
\node at (6,-0.5) {7};
\node at (8,2.5) {8};
\node at (8,-0.5) {9};
\node at (10,-0.5) {10};
\draw[thick] (-2,0) -- (0,0) -- (2,0) -- (0,2);
\draw[thick] (2,0) -- (4,0) -- (4,2);
\draw[thick] (4,0) -- (6,0) -- (8,2);
\draw[thick] (6,0) -- (8,0) -- (10,0);
\draw[dashed,thick] (0,0) -- (0,2);
\draw[dashed,thick] (8,2) -- (8,0);
\end{tikzpicture}
\caption{Graph corresponding to equation
  \eqref{eq:complex_convergence}.}
\label{fig:graph_Laplacian}
\end{figure}
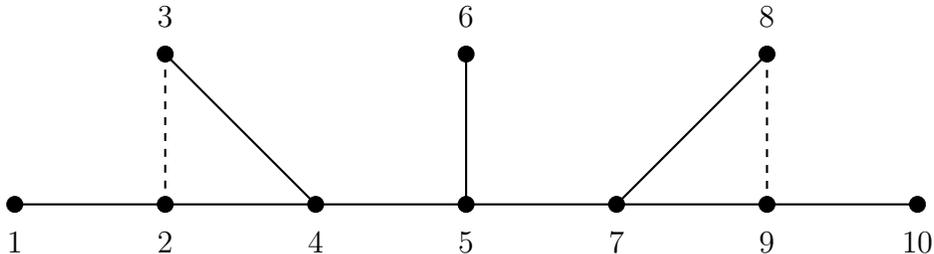

\subsection{Avoided crossing}
\label{sec:avoided}

While a non-degenerate conical point is stable under small
perturbations of the real symmetric matrix-function $A(x,y)$, the eigenvalue
multiplicity may be lifted by an addition of a small complex
perturbation.  It is instructive to investigate the run results of our
algorithm in this case.

\begin{figure}
  \centering
  \includegraphics[width=0.4\textwidth]{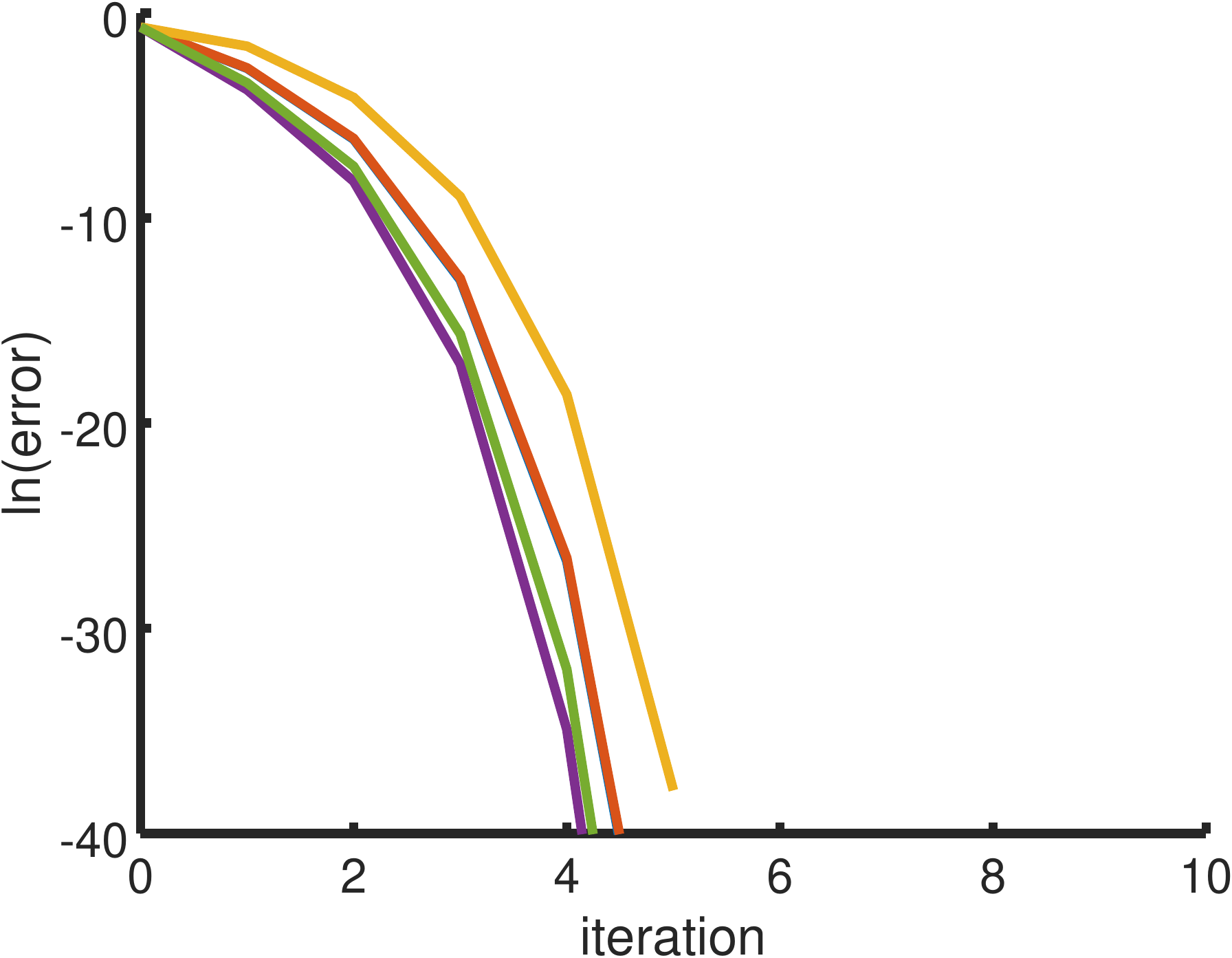}
  \quad
  \includegraphics[width=0.4\textwidth]{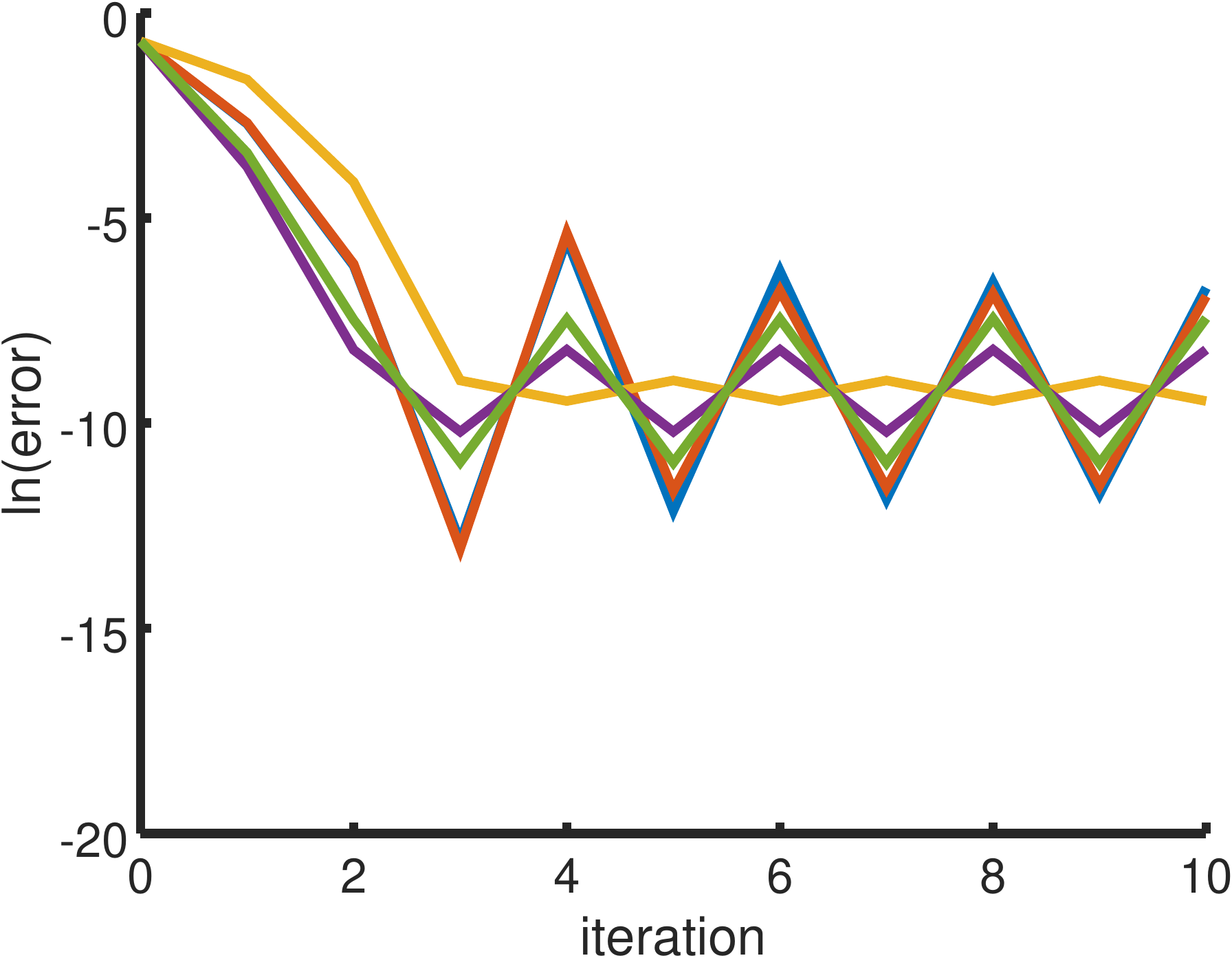}  
  \caption{Logarithm of distance to $(0,0)$ as a function of the
    iteration step for several runs of the algorithm for $A(x,y)$
    given by equation~\eqref{eq:avoided} with $\epsilon=0$ (left) and
    with $\epsilon=10^{-4}$ (right), i.e.\ an avoided crossing.  Note
    the difference in vertical scales. Runs are initialized with
    random points on the circle of radius $1/2$ around $(0,0)$.}
  \label{fig:avoided}
\end{figure}

Consider the matrix-function
\begin{equation}
  \label{eq:avoided}
  A =
  \begin{pmatrix}
    x+3\sin(y) & y+\epsilon i\\
    y-\epsilon i & -x-x^2
  \end{pmatrix}.
\end{equation}
It has a conical point at $(0,0)$ when $\epsilon=0$ and no eigenvalue
multiplicities when $\epsilon\neq 0$.  We plot in
Figure~\ref{fig:avoided} the results of several runs with $\epsilon=0$
(left) and with $\epsilon = 10^{-4}$ (right).  For $\epsilon=0$ the
algorithm converges quadratically, as in the previous examples.  For
$\epsilon\neq0$, the algorithm initially approaches the position of
the former conical point, but gets repelled, resulting in
oscillations.  Conversely, such oscillations (within the limits of
numerical precision) should be considered a tell-tale sign of
eigenvalue surfaces nearly but not exactly touching.

We remark that for $\epsilon \neq 0$, the square eigenvalue difference
$(\lambda_1 - \lambda_2)^2$ has the minimal value of order
$\epsilon^2$.  If one is using optimization of
$(\lambda_1 - \lambda_2)^2$ to find the multiplicity location, it
would be difficult to tell apart genuine points of multiplicty from
avoided crossings.  This observation is investigated further in the
next example.

\subsection{Merging Dirac points}
\label{sec:mergingDirac}

In condensed matter physics literature, the conical points in the
dispersion relation of a periodic structure are know as the ``Dirac
points'', because the effective equation of the wave propagation at
the corresponding energy is of Dirac type (see \cite{FefWei_cmp14} for
a mathematical formulation of this physics result).  When the
material parameters change, the Dirac point may undergo a fold
bifurcation, where two points collide and annihilate.  The physical
consequence of this collision were studied, for example, in
\cite{Mon+_prb09}; an experimental observation in a tunable honeycomb
lattice was reported in \cite{Tar+_n12}.  In this section we use the
basic model from \cite{Mon+_prb09},
\begin{equation}
  \label{eq:Agraphene}
  A(x,y) :=
  \begin{pmatrix}
    0 & -1 - \frac12 e^{ix} - p e^{iy}   \\
    -1 - \frac12 e^{-ix} - p e^{-iy} & 0
  \end{pmatrix},
\end{equation}
where the bifurcation occurs at $p=\frac12$: for $p>\frac12$ there are
two Dirac points and for $p<\frac12$ there are none, see
Fig.~\ref{fig:mergingDirac_surf}. 

Despite $A(x,y)$ being a complex matrix, the problem of locating Dirac
points in this setting is analogous to the real symmetric case due to
presence of the inversion symmetry $\overline{A(-x, -y)} = A(x,y)$.
The correct target function $F$ (cf.\ \eqref{eq:target_function_def}
and \eqref{eq:objective_complex_ri}) is
\begin{equation*}
  F(A^\pvec) :=
  \begin{pmatrix}
    A^{\pvec}_{11}-A^{\pvec}_{22}\\
    2\Imag(A^{\pvec}_{12})
  \end{pmatrix}.
\end{equation*}

\begin{figure}
  \centering
  \includegraphics[scale=0.5]{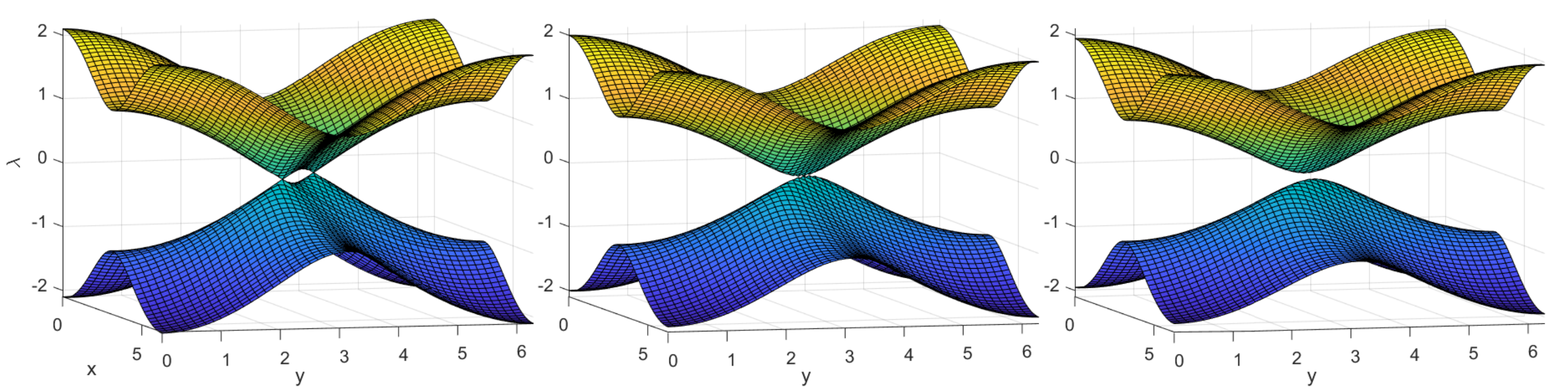}
  \caption{Two Dirac points (left), colliding (center) and disappearing
    (right), in the dispersion relation of \eqref{eq:Agraphene} with
    parameter $p$ values $0.6$, $0.5$ and $0.45$ correspondingly.}
  \label{fig:mergingDirac_surf}
\end{figure}

\begin{figure}
  \centering
  \includegraphics[scale=0.3]{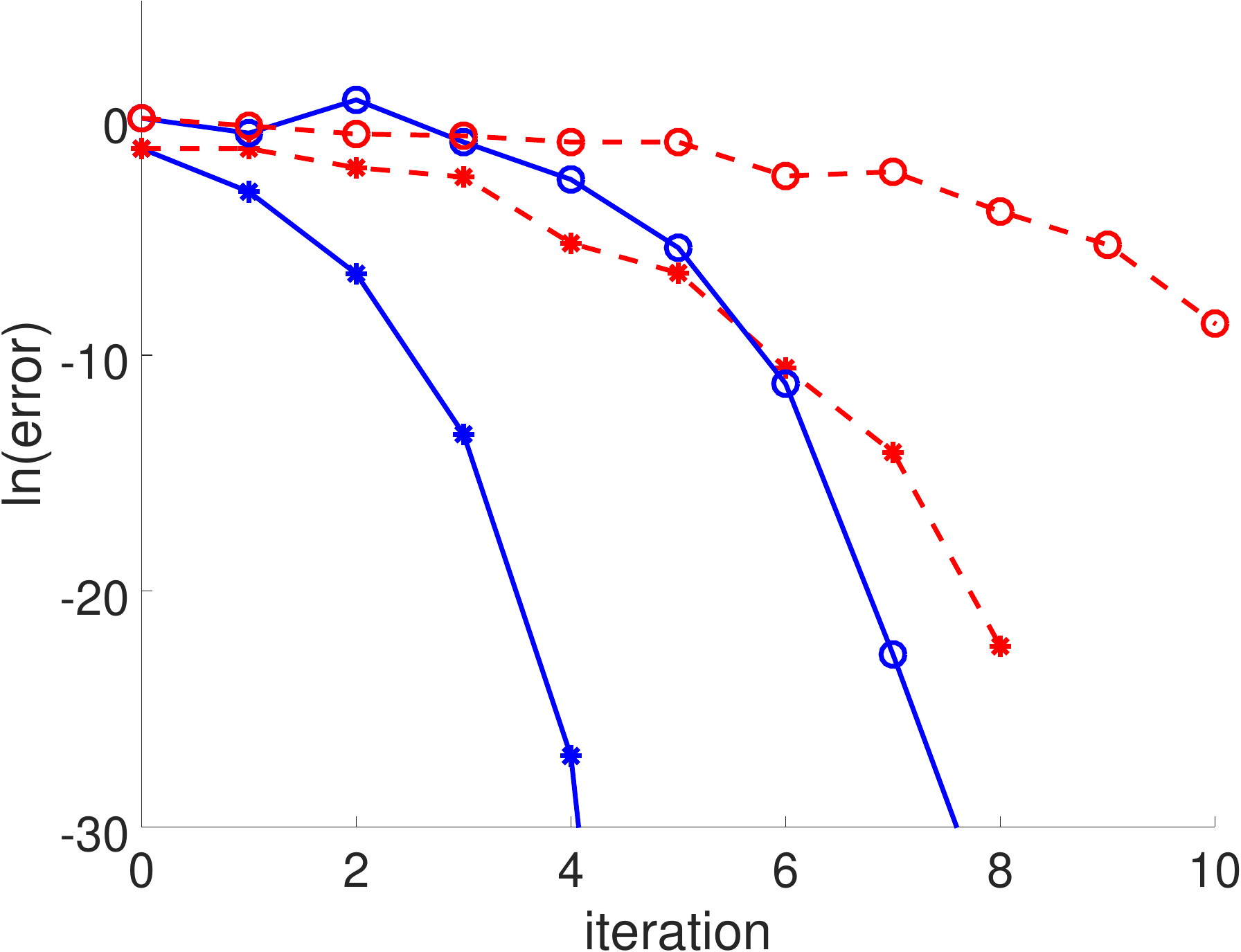}
  \includegraphics[scale=0.3]{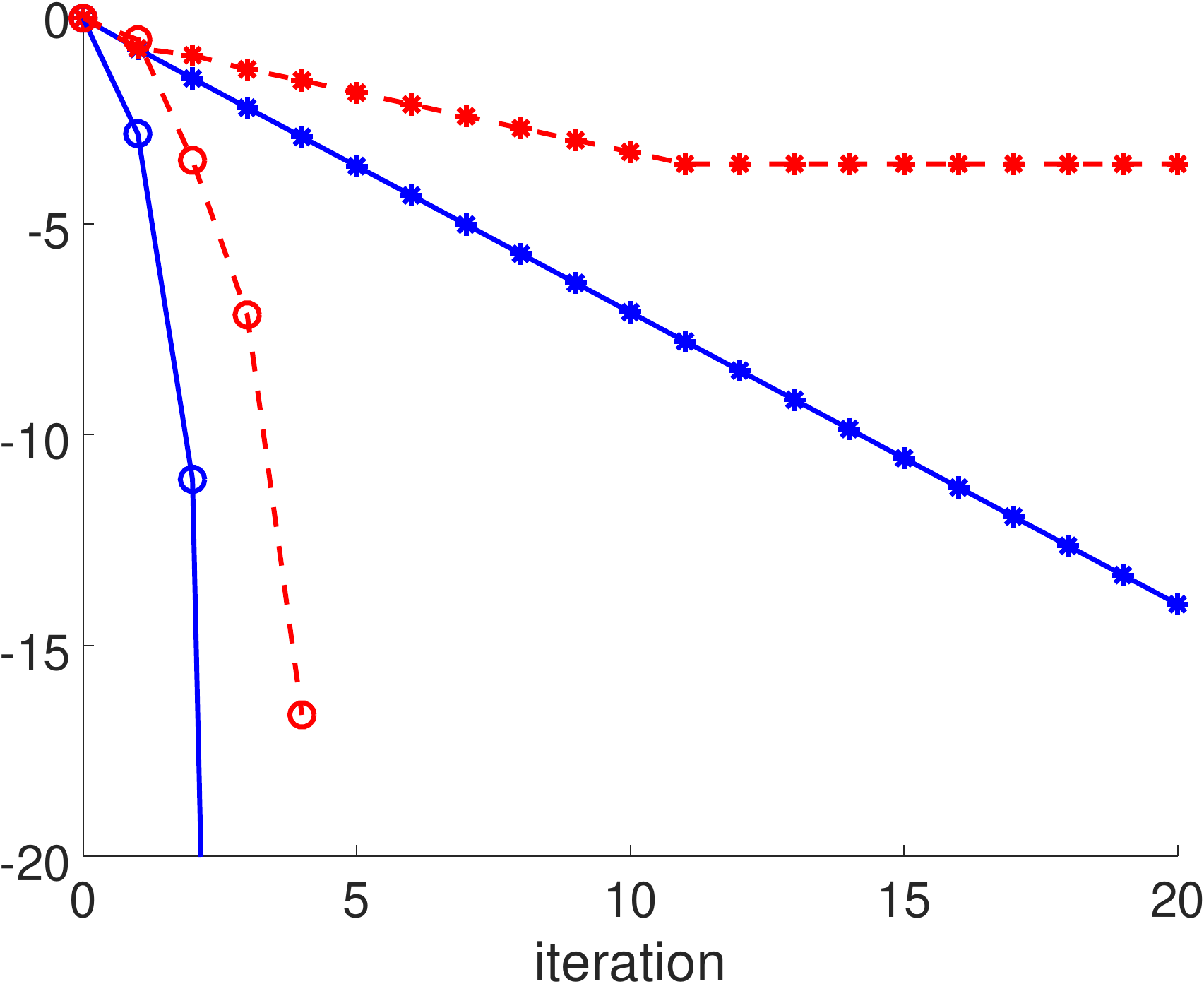}
  \includegraphics[scale=0.3]{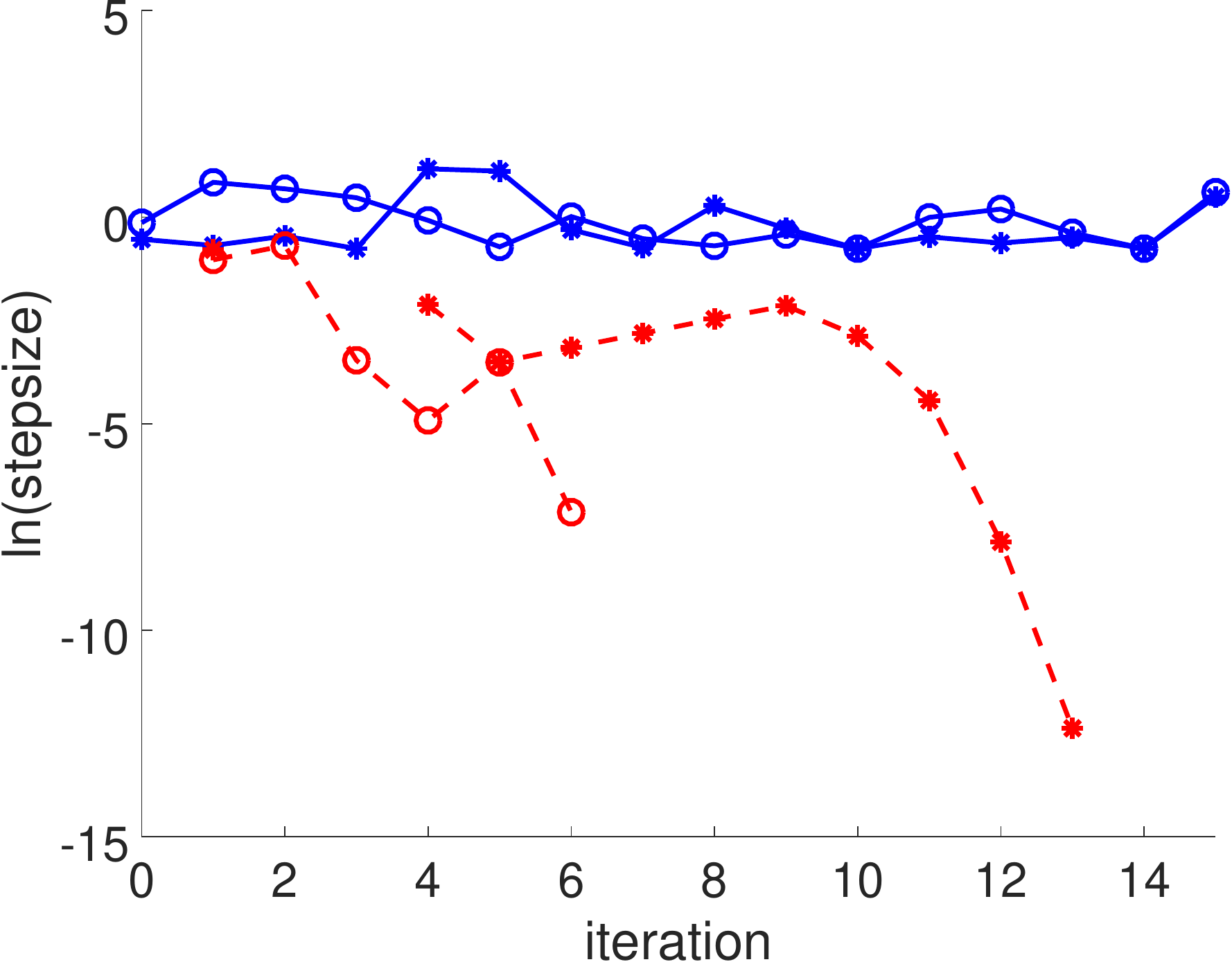}
  \caption{Convergence of iterations for matrix family
    \eqref{eq:Agraphene}: applying Theorem~\ref{thm:main} (blue,
    solid) versus quasi-Newton minimization of
    $(\lambda_1-\lambda_2)^2$ (red, dashed).  The parameter $p$ is
    $0.6$ (left), $0.5$ (center) and $0.45$ (right).  Two starting
    points used in each figure, $(0.8\pi,0.8\pi)$ (empty circles) and
    $(0.8\pi, 1.2\pi)$ (stars).}
  \label{fig:f_mergingDirac}
\end{figure}

In Figure~\ref{fig:f_mergingDirac} we present a comparison between the
convergence of iterations of Theorem~\ref{thm:main} and a standard
quasi-Newton search for the minimum of
$g(x,y) = (\lambda_1 - \lambda_2)^2$.
Figure~\ref{fig:f_mergingDirac}(left) is for $p=0.6$ where the
convergence of both methods is quadratic, although
Theorem~\ref{thm:main} is faster.
Figure~\ref{fig:f_mergingDirac}(center) is for $p=0.5$, where the
multiplicity point is \emph{degenerate}.  While Theorem~\ref{thm:main}
is no longer applicable, the iteration still converges when the matrix
pseudoinverse is used in \eqref{iter_equation_simple}.  The speed of
iteration is highly dependent on the direction, presumably because the
cross-section of the eigenvalue surface is parabolic in one direction
and conical in the other.  Again, the algorithm of
Theorem~\ref{thm:main} converges faster, while quasi-Newton iteration
fails altogether for the second initial point.

Finally, in Figure~\ref{fig:f_mergingDirac}(right), the $Y$-axis shows
the logarithm of the last taken step, since the distance to the
conical point is undefined: there is no conical point.  While the
quasi-Newton iteration converges, correctly, to the minimum of
$(\lambda_1-\lambda_2)^2$ located at $(\pi,\pi)$, the algorithm of
Theorem~\ref{thm:main} is not converging, indicating the absence of
the conical point in that area.

To interpret the results, recall that a quasi-Newton minimization is
searching for the zero of the gradient of $g$ using a numerical
approximation of the Hessian of $g$.  But according to
Theorem~\ref{thm:Hess_disc}, the matrix appearing in
equation~\eqref{iter_equation_simple} is \emph{equal} to the leading
term of the Hessian (or its square root) around the
conical point.  It is therefore natural to expect a faster
convergence.

To give an analogy, consider finding the root of $f(x)=x^2-a$ via the
Newton--Raphson scheme (thus computing $f'$ as done in
Theorem~\ref{thm:main}) or via minimization of $g(x)=f^2(x)$
(thus computing $g''$ in the course of finding the root of $g'$).  Of
course, close to the root, $g'' \approx (f')^2$, so the two schemes
give equivalent rates of convergence, but having an analytical
expression for $f'$ naturally produces better results than performing
a numerical approximation of $g''$.

Theorem~\ref{thm:main}) would thus be beneficial in any situation
where computing two eigenvectors is not significantly more expensive
than sampling the eigenvalues several times.\footnote{In the
  quasi-Newton experiment above, the eigenvalues were computed 5 times
  per iteration step in order to estimate the Hessian} One example of
such circumstances is given by differential operators on metric graphs
\cite{Ber_gcst17}, where the eigenvalues are found by solving the
``secular equation'' of the form
$\det\left(I-S(\sqrt{\lambda})\right)=0$, and, once an eigenvalue is
identified, the corresponding eigenvector of $S(\sqrt{\lambda})$ gives
the (Fourier coefficients of the) eigenvector on the graph.  The
latter operation is inexpensive relative to repeated evaluation of the
determinant necessary for locating the root $\lambda$.

\subsection{Locating points of higher multiplicity}

We can apply a modification of the method to search for points of
higher multiplicity in a family of matrices with sufficient number of
parameters.  For example, for locating a triple eigenvalue of a
5-parameter family $A$ we use
\begin{equation}
  \label{eq:F_triple}
  F(A^\pvec) =
  \begin{pmatrix}
    A^\pvec_{11} - A^\pvec_{22} \\
    A^\pvec_{22} - A^\pvec_{33} \\
    2 A^\pvec_{12} \\
    2 A^\pvec_{13} \\
    2 A^\pvec_{23}    
  \end{pmatrix},
\end{equation}
where $A^\pvec$ is the function $A(\cdot)$ expressed in the eigenbasis
calculated at point $\pvec$; the first three eigenvectors are assumed
to correspond to the consecutive eigenvalues whose point of coalescing
we are seeking.  As before,
$J_{\rvec}(A^{\pvec}) = \Diff_\rvec F(A^{\pvec})$, and a point
$\alpha$ of triple multiplicity is \emph{non-degenerate} if
$\det J_\alpha(A^\alpha) \neq 0$.

\begin{figure}
  \centering
  \includegraphics[width=0.4\textwidth]{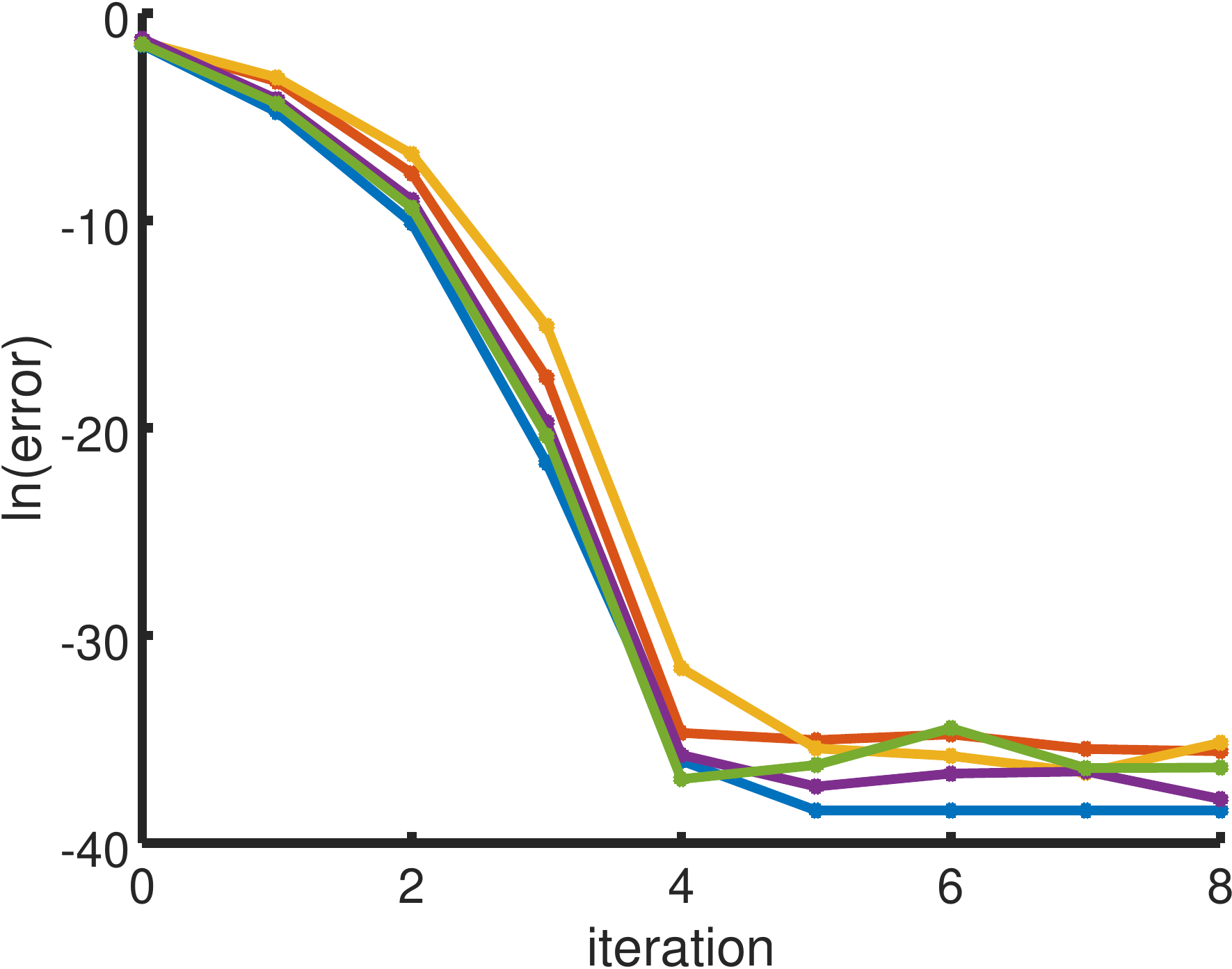}  
  \caption{Logarithm of distance to $\mathbf{0}$ as a function of the
    iteration step for several runs of the algorithm for
    $A(x,y,z,u,v)$ given by equation~\eqref{eq:5_params}. Several
    independent runs are plotted, each beginning at a random point in
    $[-0.2, 0.2]^5$.}
  \label{fig:triple}
\end{figure}

To demonstrate the performance of our method in locating a triple
multiplicity, we consider the function
\begin{equation}
  \label{eq:5_params}
  A =
  \begin{pmatrix}
    1+v+w+x-3y-z & 2x+y+2z & x+xz+y \\
    2x+y+2z & 1+x+yz & 2v-w+z \\
    x+xz+y & 2v-w+z & 1+vw
  \end{pmatrix}
\end{equation}
with triple eigenvalue at $(0,0,0,0,0)$.  The results of several runs
are shown in Figure~\ref{fig:triple}; the convergence is clearly
quadratic until the limit of numerical precision is reached in about 4 steps.

\section*{Acknowledgment}

Work on this project was supported by the National Science Foundation
through grant DMS-1815075 and the Binational US--Israel Science
Foundation grant 2016281 while one of the authors (AP) was an
undergraduate student at Texas A\&M University.  Numerous illuminating
discussions with Igor Zelenko are gratefully acknowledged.  The
authors are particularly grateful to the referees for their deep
reading of the manuscript and several suggestions which resulted in
significant improvement of the presentation.

\bibliographystyle{abbrv}
\bibliography{conical_search,bk_bibl}

\end{document}